\documentclass[3p,10pt]{elsarticle}
\usepackage{geometry}
\usepackage{amssymb,amsmath,amsthm,upgreek,bbm,ulem,setspace}
\usepackage{color}
\newtheorem{theorem}{Theorem}
\newtheorem{lemma}[theorem]{Lemma}
\newtheorem{proposition}[theorem]{Proposition}
\newtheorem{definition}[theorem]{Definition}

\journal{}

\begin{document}

\begin{frontmatter}



\title{Two-barriers-reflected BSDE with Rank-based Data}



\author{Xinwei Feng}
\ead{xwfeng@sdu.edu.cn}
\author{Lu Wang\corref{cor1}}
\ead{202490000103@sdu.edu.cn}

\address{Zhongtai Securities Institute for Financial Studies, Shandong University, Jinan, Shandong 250100, China.}
\cortext[cor1]{Corresponding author}
\begin{abstract}
We investigate two-barriers-reflected backward stochastic differential equations with data from rank-based stochastic differential equation. More specifically, we focus on the solution of backward stochastic differential equations restricted to two prescribed upper-boundary and lower-boundary processes. We rigorously show that this solution gives a probabilistic expression to the viscosity solution of some obstacle problems for the corresponding parabolic partial differential equations. As an application, the pricing problem of an American game option is studied.
\end{abstract}




\end{frontmatter}


\section{Introduction}

 Bismut \cite{Bismut} first studied the linear backward stochastic differential equation (BSDE) while Pardoux and Peng \cite{peng0}  discussed existence and uniqueness of nonlinear BSDE's solution under some suitable assumptions. El Karoui at al. \cite{peng} first considered BSDE with continuous reflection BSDE (RBSDE) and proved the existence and uniqueness of solution by two different ways (a fixed point argument and approximation via penalization).
Lepeltier and Xu \cite{Lepeltier} studied BSDE reflected on one discontinuous barrier through penalization method. Hu and Ren \cite{HR} established the relative conclusions for a class of RBSDE related to the subdifferential operator of a lower semi-continuous convex function, driven by Teugels martingales associated with a L\'{e}vy process. Luo \cite{Luo} considered a time-delayed BSDE constrained by a given barrier. Chen and Feng \cite{feng1} studied the RBSDE with data based on ranking, in which the terminal value and generator function rely on the solution to rank-based
stochastic differential equation (SDE). For BSDE restricted by two reflecting barriers (two-barriers-reflected BSDE), Cvitanic and Karatzas \cite{Karatzas} probed into the characters of solution and studied its relation with the classical Dynkin game. In Cvitani\'{c} and Karatzas \cite{Karatzas}, under Mokobodzki condition or some regularity conditions on the barrier, the existence of the solution was proved by Picard iteration and the uniqueness was gotten through penalization method.
Fan \cite{Fan} established some
results on $L^1$ solutions of two-barriers-reflected-BSDE under general assumptions and a necessary condition. Li and Shi \cite{LS} studied two-barriers-reflected BSDE  via penalization method.
In Hamad\`{e}ne and Hassani \cite{Hamad¨¨ne}, they gave a more checkable condition, i.e., $L_{t}<U_{t}$, a.s., and established existence and uniqueness under this condition.

It is well known that BSDE serves as a vital important role in many fields, such as partial differential equation (PDE), financial mathematics and so on. Combining BSDE with SDE, it will generalize the classical Feynman-Kac formula and give a probability representation of viscosity solution to a large class of semi-linear PDE. Please see \cite{Pardoux,PardouxPeng1992} for relative applications in PDE and \cite{Asri,peng,Hamad¨¨ne} for the obstacle problem.
As for applications in financial mathematics, El Karoui, Peng and Quenez \cite{EPQ} used BSDE to study the price of a contingent claim and showed that a unique hedging portfolio exists.
Furthermore, one barrier RBSDE can be used to study  American option pricing problem and please refer to \cite{EQ} for more information.
For an American game option, the trader needs to pay a premium and has the power to demand a contingent claim from the broker, when he (or she) makes up his (or her) mind to exercise within a period of time. Therefore, two-barriers-reflected BSDE provides an ideal way to study American game option.
In Ma and Cvitani\'{c} \cite{Ma}, the authors considered the case for a large investor and demonstrated that $Y_{0}$ is the value of option, in which $Y_{t}$ is the value function for the associated Dynkin game.
Kifer \cite{Kifer} paid attention on American game option problem grounded on the classical Black Scholes model. Hamad\`{e}ne \cite{Hamadene-option} studied the mixed zero-sum stochastic differential game problem with help of two-barriers-reflecting BSDE. Essaky and Hassani \cite{Essaky} discussed the relative American game option pricing problem with a new general payoff.
Please refer to \cite{chen1,HLW} and the reference therein for more relevant results on BSDE.

In above mentioned literatures, the stock prices are assumed to be geometric Brownian motion.
However, the assumption of constant volatility is unrealistic in practical applications. Stock prices are affected by many factors and the volatility changes over time.  The volatility and growth rate of stock prices are usually related to the size of capital. In other words, the fluctuations on stock price of large capital assets companies are usually different from those of companies with smaller assets. Moreover, stochastic portfolio theories (see Fernholz   and Karatzas \cite{FK2009}) have shown that the rank-based models are more consistent with the capital distribution curve in the real capital market.
 Therefore, rank-based SDEs, where
the drift and diffusion coefficients of each component are determined by its rank in the
system, are more appropriate models for stock prices and capital distributions in financial market. For rank-based SDEs, \cite{SDE4} established the existence and pathwise uniqueness of strong solution for two-dimensional case. It was extended to finite and countably infinite systems in \cite{HLW}.
 Ichiba \cite{IPS2013} studied the rates of convergence of rank-based interacting diffusions and semimartingale reflecting Brownian motions to equilibrium and worked out various applications to the rank-based abstract equity markets used in stochastic portfolio theory.
 For more information on rank-based SDE,  please refer to \cite{CP2010,SDE5,S2011} for wellposedness and \cite{SDE7,IPBKF,JR2015} for its applications in financial mathematics.

Motivated by these, in subsequent sections, we will concentrate on two-barriers-reflected BSDE linked to rank-based SDE.
In rank-based SDE, the drift and diffusion coefficients for every component are determined by the rank of the components, therefore this is a reasonable model for rank-based prices processes in reality.  Then we apply the theoretical results to study American game option with rank-based stock prices. Moreover, we also show that
the solution to this special two-barriers-reflected BSDE represents the unique viscosity solution to some obstacle problems with the corresponding parabolic PDE.

The structure of this paper is as follows. In Section 2, we firstly present some preliminaries and notations on rank-based SDEs. The main results of two-barriers-reflected BSDE with data based on ranking are introduced in Section 3. In Section 4, we study the existence and uniqueness  of viscosity solution to the related obstacle problem. Finally, we explore the American game options pricing issue in which the stock prices are related to rank-based data in Section 5.

\section{Preliminaries}

In the following, some basic notations and conclusions concerning rank-based SDE  are introduced, which will be used in the subsequent sections. On a classical complete probability space $(\Omega,\mathcal{F},\mathbb{P})$, denote by $\{\mathcal{F}_{t}^{W}\}$ the smallest filtration generated by a standard $n$-dimensional Brownian motion $\{W(t)\}$ such that all $\mathbb{P}$-null subsets of $\mathcal{F}_{\infty}^{W}$ belong to $\mathcal{F}_{0}^{W}$. Fixed a time $T>0$, we shall introduce the following spaces for $p\geq 1$ and $n\in \mathbb{N}$:

\begin{itemize}
\setlength{\itemsep}{0pt}
    \item $M^{p}\left([0,T];\mathbb{R}^{n}\right)$: the set of $\mathbb{R}^{n}$-valued progressively measurable process $\phi$ satisfying
  $\mathbb{E}\left[\int_{0}^{\mathrm{T}}\left|\phi_{t}\right|^{p} d t\right]<\infty$.
    \item $S^{p}\left([0,T] ; \mathbb{R}^{n}\right)$: the set of $\mathbb{R}^{n}$-valued progressively measurable process $\phi$ satisfying $\mathbb{E}\left[\sup\limits_{0 \leq t \leq T}\left|\phi_{t}\right|^{p}\right]<\infty$.
    \item $A_{c}\left([0,T] ; \mathbb{R}^{n}\right)$: the set of $\mathbb{R}^{n}$-valued non-decreasing adapted continuous process $\xi$
        satisfying $\xi_{0}=0$.
\end{itemize}

Besides the above spaces, there are still some function spaces need to be defined. $C\left([0,T] \times \mathbb{R}^{n}\right)$ is referred to be the class of all continuous functions while $C^{1,2}_{b}([0,T) \times \Gamma^{n})$ represents the continuous functions which are continuously differentiable in $t$ and twice continuously differentiable in $x$, with uniformly bounded partial derivatives.
   Let $S^{n}$ represent the $n\times n$ symmetric matrices and $x'$ mean the transpose of a matrix or a vector. The inner product is defined as $a \cdot b$ with $a,b$ and $\langle A,B \rangle$ is the quadratic cross-variation process of two continuous semimartingales $A,B$.  Now we introduce the following notations to study ranked data:

\begin{itemize}
\setlength{\itemsep}{0pt}
\item $\Pi^{n}:=\{x\in\mathbb{R}^{n}:x_{n}<x_{n-1}<\cdots<x_{1}\}$.
\item $F_{k}:=\{x\in\; \partial\Pi^{n}:x_{n}<\cdots<x_{k+1}=x_{k}<\cdots <x_{2}<x_{1}\}$, $k=1,\cdots,n-1$.
\item $\Gamma^{n}:= \left(F_{1}\cup F_{2}\cup \cdots \cup F_{n-1}\right)\cup \Pi^{n}$.
    \item $\partial\Pi^{n}$:\text{  the boundary of }$\Pi^{n}$.
\end{itemize}

Since we will study BSDE with ranked data, some existing results about rank-based SDE where the coefficients of this kind of SDE are associated with the ordered data are listed blow.
The specific form is listed as follows with $\sigma_{k}, \delta_{k} \in \mathbb R$ and $\delta_{k}>0$:
\begin{equation}
\begin{aligned}
\label{2.2.1}
X_{k}^{t,x}(s) =  x_{k} &+ \int_t^{s} \sum_{l=1}^n \sigma_{l} \textbf{1}_{\{X_{k}^{t,x}(u)=X_{(l)}^{t,x}(u)\}}
dW_{k}(u)+ \int_t^{s}\sum_{l=1}^n\delta_{l} \textbf{1}_{\{X_{k}^{t,x}(u)=X_{(l)}^{t,x}(u)\}}
du ,\;k=1,\cdots,n,
\end{aligned}
\end{equation}
where $s\geq t$ and the ranked arrangement of $\{X_1^{t,x}(\cdot),  \cdots, X_n^{t,x}(\cdot)\}$ is referred to
\begin{equation*}
\begin{array}{rll}
X_{(n)}^{t,x}(\cdot)\leq X_{(n-1)}^{t,x}(\cdot)\leq \dots\leq X_{(2)}^{t,x}(\cdot)\leq X_{(1)}^{t,x}(\cdot).
\end{array}
\end{equation*}
For convenience, define
\begin{equation}
\label{rank-sequence}
\widetilde{X}^{t,x}:=\{X_{(1)}^{t,x},X_{(2)}^{t,x},  \cdots, X_{(n)}^{t,x}\},
\end{equation} as ranked sequences.
Moreover, the ranked sequences admit the following formulations (see \cite{Banner}):
\begin{equation*}
\begin{array}{rll}
dX_{(l)}^{t,x}(s)=\delta_{l}ds+\sigma_{l}d\beta_{l}(s)+\frac{1}{2}d\Lambda_{l,l+1}(s)-\frac{1}{2}d\Lambda_{l-1,l}(s),\;s\geq t,\;l=1,\cdots,n,
\end{array}
\end{equation*}
in which
\begin{equation}
\label{2.1}
\beta_{l}(\cdot):=\sum_{k=1}^{n}\int_{0}^{\cdot}\textbf{1}_{\{X_{k}^{t,x}(u)=X_{(l)}^{t,x}(u)\}}dW_{k}(u),\;l=1,\cdots,n,
\end{equation}
and

\[
\Lambda_{0,1}(s)=\Lambda_{n,n+1}(s)=0,\;\Lambda_{l,l+1}(s):=D_{l}(s)-D_{l}(0)-\int_{0}^{s}\textbf{1}_{\{D_{l}(u)>0\}}dD_{l}(u),\;l=1,\cdots,n-1,\;
\]

with
\[
D_{l}(\cdot)=X_{(l)}^{t,x}(\cdot)-X_{(l+1)}^{t,x}(\cdot),\;l=1,\cdots,n-1.
\]

 Furthermore, it follows from \cite{Banner} that  \eqref{2.1} are independent Brownian motions. We notice that the $\Lambda_{l,l+1}(s)$ is the local time accumulated at the origin over $[0,s]$ by the nonnegative semimartingales $D_{l}(\cdot)$,\;$l=1,\cdots,n-1$.
 Given the assumption that the diffusion coefficients $\{\sigma_{1}^{2},\cdots,\sigma_{n}^{2}\}$ satisfies $ \frac{1}{2}(\sigma_{i}^{2}+\sigma_{i+2}^{2}) \leq \sigma_{i+1}^{2} $ for $i=1,\ldots,n-2$, it follows from \cite[Theorem 2]{SDE5} and \cite[Theorem 1.4]{SDE6} that \eqref{2.2.1} has a unique strong solution.
In the end, we set some assumptions on functions $G,g,h$, which will be used throughout the paper. \\[10pt]
\qquad(\textbf{A1)} $G:[0,T] \times \Gamma^{n}\times \mathbb{R}\times \mathbb{R}^{n}\rightarrow \mathbb{R}$ is a continuous function and there is a positive constant $c$ such that for $(t,x)$ in $[0,T] \times \Gamma^{n}$, $y_{1},y_{2} $ in $ \mathbb{R}$ and $z_{1},z_{2}$ in $ \mathbb{R}^{n}$,
{\begin{equation}
\begin{matrix}
|G(t,x,y_{1},z_{1})-G(t,x,y_{2},z_{2})|\leq c(|y_{1}-y_{2}|+|z_{1}-z_{2}|), \vspace{1ex}\\
|G(t,x,0,0)|\leq c(|x|+1).
\end{matrix}
\end{equation}
(\textbf{A2)} For $g: \Gamma^{n} \rightarrow\mathbb R$ and $h:[0,T] \times \Gamma^{n}\rightarrow \mathbb{R}$, there is a positive constant $c$ such that for any $t\in[0,T]$, $x_{1},x_{2}$ in $\Gamma^{n}$,
\begin{equation}
|g(x_{1})-g(x_{2})|\leq c|x_{1}-x_{2}|,
\end{equation}
\begin{equation}
h(t, x) \leq c\left(1+|x|^{q}\right).
\end{equation}}
\section{Two-barriers-reflected BSDE with Rank-based Data}

From now on, let us consider two-barriers-reflected BSDE combining with ranked SDE $\eqref{2.2.1}$. Besides the basic assumptions $\textbf{(A1)$-$(A2)}$, we need to add a barrier assumption $\textbf{(A3)}$ on functions $L$ and $U$.
\\[10pt]
 (\textbf{A3)}  $L:[0,T]\times \Gamma^{n}\rightarrow\mathbb R $, $U:[0,T]\times \Gamma^{n}\rightarrow\mathbb R$ are continuous functions and there is a positive constant $c$ for $(t,x)$ in $[0,T]\times \Gamma^{n}$ such that
 \begin{equation}\begin{matrix}
L(t,x)<U(t,x),\vspace{1ex}\\
L(T,x)\leq g(x)\leq U(T,x),\vspace{1ex}\\
|L(t,x)|+|U(t,x)|\leq c(1+|x|).
\end{matrix}\end{equation}\\
\indent For simplicity, let $L(\cdot):=L(\cdot,\widetilde{X}^{t,x}(\cdot))$ and $U(\cdot):=U(\cdot,\widetilde{X}^{t,x}(\cdot))$ in the subsequent content with $\widetilde{X}^{t,x}(\cdot)$ defined in \eqref{rank-sequence}.
Denote by $K^{t,x,+}(\cdot):=(K^{t,x}(\cdot))^{+}$ and $K^{t,x,-}(\cdot):=(K^{t,x}(\cdot))^{-}$ two increasing continuous processes with $K^{t,x,+}(0)=K^{t,x,-}(0)=0$.
For each initial point $(t,x)$ in $[0,T] \times \Gamma^{n}$, considering this two-barriers-reflected BSDE with the notations in \textit{Section 2}:
\begin{equation}
\label{3.1}
\begin{aligned}
Y^{t,x}(s)&=g(\widetilde{X}^{t,x}(T))+\int_{s}^{T}G(u,\widetilde{X}^{t,x}(u),Y^{t,x}(u),\hat{Z}^{t,x}(u))du-\int_{s}^{T}Z^{t,x}(u)\cdot dW(u) \\
&+K^{t,x,+}(T)-K^{t,x,+}(s)-(K^{t,x,-}(T)-K^{t,x,-}(s)),\ t\leq s\leq T.
\end{aligned}
\end{equation}
where
\[
\hat{Z}_{k}^{t,x}(\cdot):=\sum_{l=1}^{n}Z_{l}^{t,x}(\cdot)\textbf{1}_{\{X_{l}^{t,x}(\cdot)=\widetilde{X}_{(k)}^{t,x}(\cdot)\}},\;k=1,\cdots,n.
\]
Similar to \cite[Section 2.3]{feng2}, the stochastic integral part in (\ref{3.1}) is equivalent with another form:
\begin{eqnarray}
\label{3.2}
\begin{aligned}
\int_{s}^{T}Z^{t,x}(u)\cdot dW(u)=\sum_{k=1}^{n}\int_{s}^{T}\hat{Z}_{k}^{t,x}(u)d\beta_{k}(u)=\int_{s}^{T}\hat{Z}^{t,x}(u)
\cdot d\beta(u).
\end{aligned}
\end{eqnarray}
where $\beta(\cdot)$ is the vector of $\{\beta_{l}(\cdot),\;l=1,\cdots,N\}$ defined in \eqref{2.1}.
Therefore, we have the following results concerning the wellposedness of \eqref{3.1}.

\begin{theorem}\label{th1} With assumptions (\textbf{A1})-(\textbf{A3}), there is a unique quadruplet $(Y^{t,x}(\cdot),Z^{t,x}(\cdot),K^{t,x,+}(\cdot),K^{t,x,-}(\cdot))$ satisfies the properties:
\begin{enumerate}[(i)]
\setlength{\parsep}{0ex} 
\setlength{\itemsep}{0ex}
\item $Y^{t,x}(\cdot)\in S^{2}([t,T];\mathbb{R})$ and $Z^{t,x}(\cdot)\in M^{2}([t,T];\mathbb{R}^{n})$;
\item $(Y^{t,x}(\cdot),Z^{t,x}(\cdot),K^{t,x,+}(\cdot),K^{t,x,-}(\cdot))$ satisfies $\eqref{3.1}$;
\item $L(r)\leq Y^{t,x}(r) \leq U(r)$,\;$r \in [t,T]$;
\item $\int_{t}^{T}\big[Y^{t,x}(r)-L(r)\big]dK^{t,x,+}(r)=\int_{t}^{T}\big[ U(r)-Y^{t,x}(r)\big]dK^{t,x,-}(r)=0$.
\end{enumerate}
\end{theorem}
\begin{proof}
Owing to \eqref{3.2}, \eqref{3.1} can be rewritten as
\begin{equation}
\label{3.3}
\begin{aligned}
Y^{t,x}(s)&=g(\widetilde{X}^{t,x}(T))+\int_{s}^{T}G(u,\widetilde{X}^{t,x}(u),Y^{t,x}(u),\hat{Z}^{t,x}(u))du
-\int_{s}^{T}\hat{Z}^{t,x}(u)
\cdot d\beta(u) \\
&+K^{t,x,+}(T)-K^{t,x,+}(s)-(K^{t,x,-}(T)-K^{t,x,-}(s)),\ t\leq s\leq T.
\end{aligned}
\end{equation}
The existence and uniqueness of (\ref{3.3}) can be obtained  from \cite{Hamad¨¨ne}. Due to the equivalent relation between (\ref{3.1}) and (\ref{3.3}), we can obtain that $(Y^{t,x}(\cdot),Z^{t,x}(\cdot),K^{t,x,+}(\cdot),K^{t,x,-}(\cdot))$ is a unique solution of (\ref{3.1}) with two barriers $L(r)$ and $U(r)$.
\end{proof}
The unique quadruplet $(Y^{t,x}(\cdot),Z^{t,x}(\cdot),K^{t,x,+}(\cdot),K^{t,x,-}(\cdot))$ in \textit{Theorem \ref{th1}} is called the solution of \eqref{3.1}.
For a point $(t,x)$ in $[0,T]\times \Gamma^{n}$, let us give a definition of function $u(t,x)$ as
\begin{equation}\label{u}
u(t,x):=Y^{t,x}(t).
\end{equation}
In the next proposition, we will present the continuity of function $u(\cdot,\cdot)$.

\begin{proposition} With assumptions (\textbf{A1})-(\textbf{A3}), we have $u(\cdot,\cdot)\in C([0,T] \times \Gamma^{n})$.
\end{proposition}

\begin{proof}
We will take the method of penalization technique to study the continuity of $u(t,x)$. To start with, consider this BSDE associated with one-lower barrier for $s$ varying in $[t,T]$:
\begin{eqnarray}\label{penalty-1}
\begin{aligned}
\left\{\begin{array}{l}
\uline Y^{t,x,n}(s)=g(\widetilde{X}^{t,x}(T))+\displaystyle \int_{s}^{T}\Big[G(u,\widetilde{X}^{t,x}(u),\uline Y^{t,x,n}(u),\uline {\hat{Z}}^{t,x,n}(u))-n(U(u)-\uline Y^{t,x,n}(u))^{-}\Big]du \\
\qquad \qquad +(K^{t,x,n,+}(T)-K^{t,x,n,+}(s))-\displaystyle\int_{s}^{T} \uline Z^{t,x,n}(u)\cdot dW(u), \\
\uline Y^{t,x,n}(s)\geq L(s),\\
\displaystyle \int_{t}^{T}\big[\uline Y^{t,x,n}(u)-L(u)\big]dK^{t,x,n,+}(u)=0.
\end{array}\right.
\end{aligned}
\end{eqnarray}

It has been shown in \cite[Theorem 3.1]{feng2} that for any $n\in\mathbb N$, \eqref{penalty-1} admits a unique triple $(\uline{Y}^{t,x,n},\uline{Z}^{t,x,n},K^{t,x,n,+})$ and the deterministic function $\uline{u}^{n}(t,x)$ defined by
\[
\uline{u}^{n}(t,x):=\uline{Y}^{t,x,n}(t),
\]
is continuous.
Owing to the relevant comparison conclusion of BSDE with a lower barrier (see \cite[Corollary 1.4]{Hamad¨¨ne}), the solution sequence $\{\uline{Y}^{t,x,n}\}_{n\geq 0}$ is decreasing and converges to $Y^{t,x}$. From these facts, the sequence $\{\uline{u}^{n}(t,x)\}_{n\geq 0}$ decreasingly converges to the limit $u(t,x)$. This means that $u(t,x)$ is lower semi-continuous.

Next consider the following one-upper reflecting barrier BSDE for $t \leq s \leq T$:
\begin{eqnarray}\label{penalty-2}
\begin{aligned}
\left\{\begin{array}{l}
\overline Y^{t,x,n}(s)=g(\widetilde{X}^{t,x}(T))+\displaystyle\int _{s}^{T}\Big[G(u,\widetilde{X}^{t,x}(u),\overline Y^{t,x,n}(u),\hat{\overline Z}^{t,x,n}(u))+n(L(u)-\overline Y^{t,x,n}(u))^{+}\Big]du\\
\qquad \qquad -(K^{t,x,n,-}(T)-K^{t,x,n,-}(s))-\displaystyle\int_{s}^{T}\overline Z^{t,x,n}(u)\cdot dW(u), \\
Y^{t,x,n}(s)\leq U(s), \\
\displaystyle\int_{t}^{T}\big[U(u)-\overline Y^{t,x,n}(u)\big]dK^{t,x,n,-}(u)=0.
\end{array}\right.
\end{aligned}
\end{eqnarray}

It should be noted that $(-Y,-Z,K)$ is corresponding solution of a lower-reflecting barrier BSDE connected with $(-G(t,x,-y,-z),-g,-U)$ iff $(Y,Z,K)$ is a solution of a upper-reflecting barrier BSDE connected with $(G,g,U)$. Thus, the above penalization scheme can be transferred to a lower-reflecting barrier BSDE combined with $(-G(t,x,-y,-z)-n(L-y)^{+},-g,-U)$. The notation $\overline{Y}^{t,x,n}$ is defined as a solution of \eqref{penalty-2} and the sequence $\big\{\overline{Y}^{t,x,n}\big\}_{n\geq 0}$ is increasing and converges to the limit $Y^{t,x}$.
Similarly, we conclude that
 $\big\{\overline{u}^{n}(t,x):=\overline{Y}^{t,x,n}(t)\big\}_{n\geq 0}$ increasingly converges to limit $u(t,x)$. Therefore, $u$ is upper semi-continuous and the proof is complete.
\end{proof}

\section{PDEs Related to the Two-barriers-reflected BSDE}

In this part, we shall investigate the viscosity solution of the following kind of PDE:
\begin{eqnarray}
\label{4}
\left\{\begin{array}{l}
\min\left\{u(t,x)-L(t,x),\max\left\{u(t,x)-U(t,x), -\frac{\partial u}{\partial t}(t,x)-
\mathcal{L}u(t,x)-G(t,x,u(t,x),\sigma \triangledown u(t,x))\right\}\right\} =0,\\
\qquad \qquad  \qquad \qquad\qquad\qquad\qquad\quad \qquad \qquad\qquad (t,x)\in [0,T)\times \Pi ^{n},
 \\
  u(T,x)=g(x),\qquad \qquad  \;\qquad\qquad \qquad \qquad\qquad  \,x\in \Gamma^{n},
 \\
 \frac{\partial u}{\partial x_{k+1}}(t,x)=\frac{\partial u}{\partial x_{k}}(t,x),\qquad  \qquad  \qquad   \qquad  \quad \;\;\;  \quad  (t,x)\in [0,T)\times F_{k},\;k=1,\cdots,n-1,
\end{array}\right.
\end{eqnarray}
where
\begin{equation}
\label{sigma}
\sigma=\begin{pmatrix} \sigma_{1}& 0 & \cdots &0\\
0 & \sigma_{2} &\cdots & 0\\
\vdots & \vdots & \vdots & \vdots\\
0& 0& \cdots & \sigma_{n}
\end{pmatrix},
\end{equation}
and
\[
\mathcal{L}=\frac{1}{2}\sum_{i=1}^{n}\left(\sigma_{i}^{2}\frac{\partial^{2}}{\partial x_{i}^{2}}+2\delta_{i}\frac{\partial}{\partial x_{i}}\right).
\]

To begin with, we state two types of definition for viscosity solution, which will be used in the \textit{Section 4.1} and \textit{Section 4.2}.   

\begin{definition}

A function $v(\cdot,\cdot) \in C([0,T]\times \Gamma^{n})$ is called a viscosity subsolution (respectively, supersolution) to ($\ref{4}$) on condition that it satisfies for any $ x\in \Gamma^{n}$
\[
g(x)\geq v(T,x)\;\;(respectively,\;g(x) \leq v(T,x)),
\]
and for any  $\varphi(\cdot,\cdot) \in  C^{1,2}_{b}([0,T] \times \Gamma^{n})$, whenever $(t,x)\in[0,T)\times \Gamma^{n}$ is a local maximum (respectively, minimum) of $v(\cdot,\cdot) - \varphi(\cdot,\cdot)$, we have
\begin{eqnarray*}
\left\{\begin{array}{l}
\min\Big\{v(t,x)-L(t,x), \max\Big\{v(t,x)-U(t,x),\\
\qquad \qquad \qquad\quad\;\qquad-\frac{\partial \varphi}{\partial t}(t,x)-
\mathcal{L}\varphi(t,x)-G(t,x,v(t,x),\sigma \triangledown \varphi(t,x))  \Big\}\Big\} \leq 0,\;x \in \Pi^{n},\\
\min\left\{\frac{\partial \varphi}{\partial x_{k+1}}(t,x)-\frac{\partial \varphi}{\partial x_{k}}(t,x), \min\Big\{v(t,x)-L(t,x),\max\Big\{v(t,x)-U(t,x),\right.\\
\left.\left.\left.
\qquad \qquad  \qquad \quad  -\frac{\partial \varphi}{\partial t}(t,x)-\mathcal{L}\varphi(t,x)-G(t,x,v(t,x),\sigma \triangledown \varphi(t,x))  \right\}\right\} \right\} \leq 0,\; x \in F_{k},\;k=1,\cdots,n.
\end{array}\right.
\end{eqnarray*}
(respectively,
\begin{eqnarray*}
\left\{\begin{array}{l}
\min\Big\{v(t,x)-L(t,x), \max\Big\{v(t,x)-U(t,x),\\
\qquad \qquad \qquad\quad\;\qquad-\frac{\partial \varphi}{\partial t}(t,x)-
\mathcal{L}\varphi(t,x)-G(t,x,v(t,x),\sigma \triangledown \varphi(t,x))  \Big\}\Big\}\geq 0,\; x \in \Pi^{n},\\
\max\Big\{\frac{\partial \varphi}{\partial x_{k+1}}(t,x)-\frac{\partial \varphi}{\partial x_{k}}(t,x), \min\Big\{v(t,x)-L(t,x), \max\Big\{v(t,x)-U(t,x),\\
\qquad \qquad \qquad\quad\;\qquad-\frac{\partial \varphi}{\partial t}(t,x)-
\mathcal{L}\varphi(t,x)-G(t,x,v(t,x),\sigma \triangledown \varphi(t,x))  \Big\}\Big\}\Big\} \geq 0,\;x \in F_{k},\;k=1,\cdots,n.)
\end{array}\right.
\end{eqnarray*}

In all, if a function $v(\cdot,\cdot)\in C([0,T]\times \Gamma^{n})$ is both a viscosity supersolution and subsolution, it is called a viscosity solution to ($\ref{4}$).
\end{definition}
For convenience, for $v(\cdot,\cdot)\in C\left([0,T] \times \mathbb{R}^{n}\right)$, we set $\bar{\mathcal K}_{v}^{2,+}(t, x)$ as the parabolic superjet of $v(\cdot,\cdot)$ at $(t,x)$ and $\bar{\mathcal K}_{v}^{2,-}(t, x)$ as the parabolic subjet of $v(\cdot,\cdot)$ at $(t,x)$. More detailed, that is, $\bar{\mathcal K}_{v}^{2,+}(t, x)$ is the set of triples $(P,Q,M)$ defined as follows:
\begin{eqnarray}
\begin{array}{l}
\bar{\mathcal K}_{v}^{2,+}(t,x):=\Big\{(P, Q, M) \in \mathbb{R} \times \mathbb{R}^{n}\times S^{n}|\ \forall(r, z)\in[0, T] \times \mathbb{R}^{n},\\ \qquad \qquad \; v(r, z)\leq v(t, x)+Q \cdot(z-x)+P(r-t)+\frac{1}{2}(z-x)' \cdot M(z-x)+o\left(|z-x|^{2}+|r-t|\right)\Big\}.
\end{array}
\end{eqnarray}
Homoplastically, $\bar{\mathcal K}_{v}^{2,-}(t, x)$ is the set satisfying:
\begin{eqnarray}
\begin{array}{l}
\bar{\mathcal K}_{v}^{2,-}(t,x):=\Big\{(P,Q,M) \in \mathbb{R} \times \mathbb{R}^{n}\times S^{n}|\ \forall(r, z)\in[0, T] \times \mathbb{R}^{n}, \\
 \qquad \qquad \;  v(r, z) \geq v(t, x)+Q \cdot(z-x)+P(r-t)+\frac{1}{2}(z-x)' \cdot M(z-x)+o\left(|z-x|^{2}+|r-t|\right)\Big\}.
\end{array}
\end{eqnarray}

By the concept of superjet and subjet, we give another definition of viscosity solution originated from \cite{Crandall}. Please refer to \cite[Lemma 5.4, Chapter 4]{Yong_book} for the detailed proof of the equivalence. For simplicity, let $\delta:=(\delta_{1},...,\delta_{n})$ with $\delta_{1},...,\delta_{n}$ in \eqref{2.2.1} and $\sigma$ be the matrix in \eqref{sigma}.
\begin{definition}
Funcition $v(\cdot,\cdot)\in C([0,T]\times \Gamma^{n})$ is called a viscosity subsolution (respectively, supersolution) of ($\ref{4}$) on condition that  it satisfies for any $ x\in \Gamma^{n}$
\[
g(x)\geq v(T,x)\;(respectively,\;g(x) \leq v(T,x) ),
\]
and for any $(P,q,M)\in \bar{\mathcal K}_{v}^{2,+}(t,x)$, we have
\begin{eqnarray*}
	\left\{\begin{array}{l}
		\min\Big\{v(t,x)-L(t,x),\max\Big\{ v(t,x)-U(t,x),\\
\qquad\qquad \qquad \qquad\qquad \qquad-P-\delta\cdot q-\frac{1}{2}tr(\sigma^{2}M)-G(t,x,v(t,x),\sigma q)  \Big\}\Big\} \leq 0,\;x \in \Pi^{n}, \\
		\min\Big\{q_{k+1}-q_{k},\min\Big\{ v(t,x)-L(t,x),\max\Big\{v(t,x)-U(t,x), \\
\qquad\qquad \qquad \qquad\qquad \qquad-P-\delta\cdot q-\frac{1}{2}tr(\sigma^{2}M)-G(t,x,v(t,x),\sigma q) \Big\}\Big\}\Big\} \leq 0,\;x \in F_{k},\;k=1,\cdots,n.\\
	\end{array}\right.
\end{eqnarray*}
(respectively, for any $(P,q,M)\in \bar{\mathcal K}_{v}^{2,-}(t,x)$, we have
\begin{eqnarray*}
	\left\{\begin{array}{l}
		\min\Big\{v(t,x)-L(t,x),\max\Big\{ v(t,x)-U(t,x),\\
\qquad\qquad \qquad \qquad\qquad \qquad-P-\delta\cdot q-\frac{1}{2}tr(\sigma^{2}M)-G(t,x,v(t,x),\sigma q)  \Big\}\Big\} \geq 0,\;x \in\Pi^{n}, \\
		\max\Big\{q_{k+1}-q_{k}, \min\Big\{v(t,x)-L(t,x),\max\Big\{ v(t,x)-U(t,x),\\
\qquad\qquad \qquad \qquad\qquad \qquad-P-\delta\cdot q-\frac{1}{2}tr(\sigma^{2}M)-G(t,x,v(t,x),\sigma q)  \Big\}\Big\}\Big\}\geq 0,\;x \in F_{k},\;k=1,\cdots,n.)\\
	\end{array}\right.
\end{eqnarray*}

In all, if a function $v(\cdot,\cdot)\in C([0,T]\times \Gamma^{n})$ is both a viscosity subsolution and a supersolution, it is called a viscosity solution to ($\ref{4}$).
\end{definition}

Noting that $\min\{a, \max\{b,c\}\}=\max\{b,  \min\{a, c\}\}$ holds for $a\geq b$ with $a,b,c \in \mathbb{R}$, therefore,
\begin{equation}
\label{4.0.0.0}
\begin{aligned}
&\min\Big\{v(t,x)-L(t,x),\max\Big\{ v(t,x)-U(t,x),-P-\delta\cdot q-\frac{1}{2}tr(\sigma^{2}M)-G(t,x,v(t,x),\sigma q)  \Big\}\Big\}\\
=&\max\Big\{v(t,x)-U(t,x), \min\Big\{v(t,x)-L(t,x),-P-\delta\cdot q-\frac{1}{2}tr(\sigma^{2}M)-G(t,x,v(t,x),\sigma q) \Big\}\Big\}.
\end{aligned}
\end{equation}

In the subsequent part, we would demonstrate that the solution of two-barriers-reflected BSDE studied in the previous section (i.e., \eqref{u}) presents a probabilistic expression to the viscosity solution of \eqref{4}.

\hspace*{1.em}
\subsection{Existence of Viscosity Solution}
\hspace*{0.5em}
\begin{theorem}
  Assume that (\textbf{A1)}-(\textbf{A3)} are satisfied, $u(\cdot,\cdot)$ defined by ($\ref{u}$) is a viscosity solution of ($\ref{4}$).
\end{theorem}

\begin{proof}
We will only prove that $u(\cdot,\cdot)$ is a viscosity subsolution since the proof of supersolution is quite similar owing to \eqref{4.0.0.0}.
At the beginning, observe the fact that
$Y^{t,x}(s)=Y^{s,\widetilde{X}^{t,x}(s)}(s)=u(s,\widetilde{X}^{t,x}(s))$, $t\leq s\leq T$. 
 Let $\varphi \in C^{1,2}_{b}([0,T] \times \Gamma^{n})$  and $(t,x)\in[0,T)\times \Gamma^{n}$ is a local maximum of $u - \varphi$. Without loss of generality, assume that $\varphi(t,x)=u(t,x)$.

\textbf{Case 1:} $ x \in \Pi ^{n}$.\quad From the definition of $u$, we have $L(t,x)\leq u(t,x)\leq U(t,x)$. If $u(t,x)=L(t,x)$, obviously,
 $$\min\Big\{u(t,x)-L(t,x), \max\Big\{u(t,x)-U(t,x),-\frac{\partial \varphi}{\partial t}(t,x)-
\mathcal{L}\varphi(t,x)-G(t,x,u(t,x),\sigma \triangledown \varphi(t,x))  \Big\}\Big\} \leq 0.$$ Thus  we just need to consider the case $u(t,x)>L(t,x)$. Under these assumptions, in order to get
 $$\min\Big\{u(t,x)-L(t,x), \max\Big\{u(t,x)-U(t,x),-\frac{\partial \varphi}{\partial t}(t,x)-
\mathcal{L}\varphi(t,x)-G(t,x,u(t,x),\sigma \triangledown \varphi(t,x))  \Big\}\Big\} \leq 0,$$
it is
sufficient to show the following inequality holds:
\begin{eqnarray}
\label{4.1.0}
M(t,x;\varphi,u):= -\frac{\partial \varphi}{\partial t}(t,x)-
\mathcal{L}\varphi(t,x)-G(t,x,u(t,x),\sigma \triangledown \varphi(t,x)) \leq 0.
\end{eqnarray}
We will prove it by contradiction. Suppose \eqref{4.1.0} is not true. Naturally, there is a positive constant $\varepsilon$ so that
 \[
 \mu(t,x) := \min\Big\{\Big(u(t,x)-L(t,x)\Big), M(t,x;\varphi,u)\Big\} \geq 2\varepsilon.
 \]
Let $0<\alpha \leq T-t$ be such that  $\mu(r,y)\geq \varepsilon $ for every $t\leq r \leq t+\alpha $, $y\in E:=\{y:|y-x|\leq \alpha\}\subset\Pi^{n}$. Define a stopping time $\tau$ by
\[
\tau:=\min\left\{\inf\Big\{r>t:|\widetilde{X}^{t,x}(r)-x|\geq \alpha\Big\}, (t+\alpha)\right\}.
\]
 Let $K(\cdot):=K^{t,x,+}(\cdot)-K^{t,x,-}(\cdot)$, then
\begin{eqnarray}
\label{4.1.1}
\begin{aligned}
u(t,x)&=Y^{t,x}(\tau)+\int_{t}^{\tau}G(a,\widetilde{X}^{t,x}(a),Y^{t,x}(a),\hat{Z}^{t,x}(a))da+K(\tau)-K(t)-\int_{t}^{\tau}\hat{Z}^{t,x}(a)
\cdot d\beta(a) \\
&=u(\tau,\widetilde{X}^{t,x}(\tau))+\int_{t}^{\tau}G(a,\widetilde{X}^{t,x}(a),Y^{t,x}(a),\hat{Z}^{t,x}(a))da+K(\tau)-K(t)-\int_{t}^{\tau}\hat{Z}^{t,x}(a)
\cdot d\beta(a).
\end{aligned}
\end{eqnarray}
Applying It\^{o}'s formula to $\varphi(r,\widetilde{X}^{t,x}(r))$, we have
\begin{equation*}
\begin{aligned}
d\varphi(r,\widetilde{X}^{t,x}(r)) &= \frac{\partial \varphi}{\partial
r}dr+\sum\limits_{k=1}^{n}\frac{\partial \varphi}{\partial
x_{k}}dX_{(k)}^{t,x}(r)+\frac{1}{2}\sum\limits_{k=1}^{n}\sum\limits_{j=1}^{n}\frac{\partial \varphi}{\partial
x_{k}}\frac{\partial \varphi}{\partial
x_{j}}d\langle X_{(k)}^{t,x},X_{(j)}^{t,x}\rangle_{r} \\
&=\frac{\partial \varphi}{\partial
r}dr+\sum\limits_{k=1}^{n}\delta_{k}\frac{\partial \varphi}{\partial
x_{k}}dr
+\frac{1}{2}\sum\limits_{k=1}^n\sigma_{k}^{2}\frac{\partial^{2}
\varphi}{\partial x_{k}^{2}}dr
+\sum\limits_{k=1}^{n}\sigma_{k}\frac{\partial \varphi}{\partial
x_{k}}d\beta_{k}(r)\\
&-\frac{1}{2}\sum\limits_{k=1}^{n}\frac{\partial
\varphi}{\partial
x_{k}}d\Lambda_{k-1,k}(r)+\frac{1}{2}\sum\limits_{k=1}^{n}\frac{\partial
\varphi}{\partial x_{k}}d\Lambda_{k,k+1}(r).
\end{aligned}
\end{equation*}
Integrating from time t to $\tau$,
\begin{equation}
\begin{aligned}
\label{4.1.2}
\varphi(\tau,\widetilde{X}^{t,x}(\tau))&=\int_{t}^{\tau}\left[\frac{\partial
\varphi}{\partial
a}(a,\widetilde{X}^{t,x}(a))+\mathcal{L}\varphi(a,\widetilde{X}^{t,x}(a))\right]da+\sum\limits_{k=1}^{n}\int_{t}^{\tau}\sigma_{k}\frac{\partial
\varphi}{\partial x_{k}}(a,\widetilde{X}^{t,x}(a))d\beta_{k}(a)+\varphi(t,x).
\end{aligned}
\end{equation}

By the fact $u(t,x)=\varphi(t,x)$, subtracting $(\ref{4.1.2})$ from $(\ref{4.1.1})$, we get
\begin{equation*}
\begin{aligned}
u(\tau,\widetilde{X}^{t,x}(\tau)) -
\varphi(\tau,\widetilde{X}^{t,x}(\tau))
&=\int_{t}^{\tau}\left[-\frac{\partial \varphi}{\partial
a}(a,\widetilde{X}^{t,x}(a))-\mathcal{L}\varphi(a,\widetilde{X}^{t,x}(a))-G(a,\widetilde{X}^{t,x}(a),Y^{t,x}(a),\hat{Z}^{t,x}(a))\right]da\\
&-\left(K(\tau)-K(t)\right)+\int_{t}^{\tau}\Big[\hat{Z}^{t,x}(a)-\sigma \triangledown
\varphi(a,\widetilde{X}^{t,x}(a))\Big]\cdot d\beta(a).
\end{aligned}
\end{equation*}
Moreover, notice that $u(\cdot,\cdot)\leq \varphi(\cdot,\cdot)$, we have
\begin{equation*}
\begin{aligned}
&\int_{t}^{\tau}\left[-\frac{\partial \varphi}{\partial
a}(a,\widetilde{X}^{t,x}(a))-\mathcal{L}\varphi(a,\widetilde{X}^{t,x}(a))-G(a,\widetilde{X}^{t,x}(a),Y^{t,x}(a),\hat{Z}^{t,x}(a))\right]da\\
&-\left(K(\tau)-K(t)\right)+\int_{t}^{\tau}\Big[\hat{Z}^{t,x}(a)-\sigma \triangledown
\varphi(a,\widetilde{X}^{t,x}(a))\Big]\cdot d\beta(a) \leq 0.
\end{aligned}
\end{equation*}
For $t\leq a\leq \tau$, define $\hat M$ and $b$,
\begin{equation*}
\begin{aligned}
&\hat M(a):= -\frac{\partial \varphi}{\partial
a}(a,\widetilde{X}^{t,x}(a))-\mathcal{L}\varphi(a,\widetilde{X}^{t,x}(a))-G(a,\widetilde{X}^{t,x}(a),u(a,\widetilde{X}^{t,x}(a)),\sigma\triangledown
\varphi(a,\widetilde{X}^{t,x}(a))),\\
&b(a):=G(a,\widetilde{X}^{t,x}(a),u(a,\widetilde{X}^{t,x}(a)),\sigma\triangledown
\varphi(a,\widetilde{X}^{t,x}(a)))-G(a,\widetilde{X}^{t,x}(a),u(a,\widetilde{X}^{t,x}(a)),\hat{Z}^{t,x}(a)).
\end{aligned}
\end{equation*}
Since $b(a)\leq C|\triangledown \varphi(a,\widetilde{X}^{t,x}(a))-\hat{Z}^{t,x}(a)|$
 for some positive constant $C$, there is an adapted bounded process
$h(a)$ such that
\[
b(a)=h(a)\cdot \Big[\triangledown \varphi(a,\widetilde{X}^{t,x}(a))-\hat{Z}^{t,x}(a)\Big].
\]
As a consequence,
\begin{equation}
\label{4.1.3}
\begin{aligned}
u(\tau,\widetilde{X}^{t,x}(\tau)) -
\varphi(\tau,\widetilde{X}^{t,x}(\tau))&=\int_{t}^{\tau}\left[\hat M(a)-h(a)\cdot
(\hat{Z}^{t,x}(a)-\sigma\triangledown \varphi(a,\widetilde{X}^{t,x}(a))\right]da\\
&-(K(\tau)-K(t))+\int_{t}^{\tau}\Big[\hat{Z}^{t,x}(a)-\sigma \triangledown
\varphi(a,\widetilde{X}^{t,x}(a))\Big]\cdot d\beta(a)\leq0.
\end{aligned}
\end{equation}
Let $\{N^{t}(r),t\leq r\leq T\}$ be the unique solution of,
\begin{eqnarray*}
\begin{array}{l}
dN^{t}(r)=h(r)N^{t}(r) \cdot d\beta(r),\;N^{t}(t)=1,\\
\end{array}
\end{eqnarray*}
namely, for $r$ varying in $[t,T]$,
\[
N^{t}(r)=\exp\left\{-\frac{1}{2}\int_{t}^{r}|h(a)|^{2}da+\int_{t}^{r}h(a)\cdot d\beta(a)\right\},
\]
which is a $\mathbb P$-martingale. Based upon the Girsanov's Theorem, there is a new probability measure $\mathbb Q$  satisfying
$d\mathbb Q=N^{t}(T)d\mathbb{P}$ and $$\widetilde{\beta}^{t}(r)=\beta(r)-\beta(t)-\int_{t}^{r}h(a)da,$$ is a
$\mathbb Q$-Brownian motion.
And then,
\begin{equation*}
\begin{aligned}
&-\int_{t}^{\tau}\left[h(a)\cdot (\hat{Z}^{t,x}(a))-\sigma\triangledown
\varphi(a,\widetilde{X}^{t,x}(a))\right]da+\int_{t}^{\tau}\left[\hat{Z}^{t,x}(a)-\sigma\triangledown
\varphi(a,\widetilde{X}^{t,x}(a))\right]\cdot d\beta(a)\\
&=\int_{t}^{\tau}\left[\hat{Z}^{t,x}(a)-\sigma\triangledown
\varphi(a,\widetilde{X}^{t,x}(a))\right]\cdot d\widetilde{\beta}(a).
\end{aligned}
\end{equation*}
Therefore,
\begin{eqnarray}
\label{new-1}
\int_{t}^{\tau}\hat M(a)da
-(K(\tau)-K(t))+\int_{t}^{\tau}\Big[\hat{Z}^{t,x}(a)-\sigma \triangledown
\varphi(a,\widetilde{X}^{t,x}(a))\Big]\cdot d\widetilde\beta(a)\leq0.
\end{eqnarray}

Next, we verify the above integral is  a $\mathbb Q$-martingale on interval $[t,T]$. Owing to the property that $\triangledown \varphi$ is uniformly bounded,
\begin{equation*}
\begin{aligned}
&\widetilde{\mathbb E}\left\{\int_{t}^T\left|\hat{Z}^{t,x}(a)-\sigma\triangledown
\varphi(a,\widetilde{X}^{t,x}(a))\right|^{2}da\right\}^{\frac{1}{2}}\vspace{1ex}\\
&\leq C\left\{\widetilde{\mathbb E}\int_{t}^{T}\left(|\hat{Z}^{t,x}(a)|^{2}+1\right)da\right\}^{\frac{1}{2}}\\
&\leq C\Bigg\{\mathbb E\left[|N^{t}(T)|^{2}\right]\Bigg\}^{\frac{1}{2}} \left\{
\mathbb E\int_{t}^{T}\left[|\hat{Z}^{t,x}(a)|^{2}+1\right]da\right\}^{\frac{1}{2}}<\infty. \vspace{1ex}
\end{aligned}
\end{equation*}
Hence, taking expectation $\widetilde{\mathbb E}$ with respect to $\mathbb Q$ of $(\ref{new-1})$,
\begin{eqnarray}
\label{4.1.4}
\begin{aligned}
&\widetilde{\mathbb E}\int_{t}^{\tau}\Big[-\frac{\partial \varphi}{\partial
a}(a,\widetilde{X}^{t,x}(a))-\mathcal{L}\varphi(a,\widetilde{X}^{t,x}(a))-G(a,\widetilde{X}^{t,x},u(a,\widetilde{X}^{t,x}(a)),\sigma\triangledown
\varphi(a,\widetilde{X}^{t,x}(a)))\Big]da\\
-&\widetilde{\mathbb E}\Big[K(\tau)-K(t)\Big]\leq 0.
\end{aligned}
\end{eqnarray}
From the definition of $\tau$, it is naturally to see that the first integral equation is no less than $\varepsilon\widetilde{\mathbb E}(\tau-t)$.

Now turn to the the last term on the left side of \eqref{4.1.4}.
Define a process
$\gamma_{a}^{t}:=[u(a,\widetilde{X}^{t,x}(a))-\varepsilon]\textbf{1}_{[t,\tau)}(a)$, which is a
RCLL process. It also satisfies
\[
L(a,\widetilde{X}^{t,x}(a))\leq \gamma_{a}^{t}=u(a,\widetilde{X}^{t,x}(a))-\varepsilon<u(a,\widetilde{X}^{t,x}(a))\leq U(a,\widetilde{X}^{t,x}(a)),\;a\in[t,\tau).
\]
Hence, by the definition of process $K$, we obtain
\[
\varepsilon(K(\tau)-K(t))=\int_{t}^{\tau}\Big[u(a,\widetilde{X}^{t,x}(a))-\gamma_{a}^{t}\Big]dK=\int_{t}^{\tau}\Big[Y^{t,x}(a)-\gamma_{a}^{t}\Big]dK =-\int_{t}^{\tau}\Big[Y^{t,x}(a)-\gamma_{a}^{t}\Big]dK^-\leq 0,
\]
which shows the last part on the left side of \eqref{4.1.4} is non-negative. As a consequence,
\begin{equation}
\label{maodun}
\varepsilon\widetilde{\mathbb E}(\tau-t) \leq 0.
\end{equation}
This is contradictory to the definition of $\tau$.

\textbf{Case 2:} $ x \in F_{k}$ for $k=1,2,\cdots,n-1$. Similarly, assume that $u(t,x)>L(t,x)$. It then suffices to show that either $\{(\frac{\partial u}{\partial x_{k+1}}-\frac{\partial u}{\partial x_{k}})(t,x) \leq 0\}$ or $\{ M(t,x;u,\varphi) \leq 0\}$.
We still apply the method of contradiction to prove this. Suppose $0<\alpha \leq T-t$ and there is a constant $\varepsilon > 0$ such that
\[
\inf\limits_{|z-x|\leq \alpha,t\leq r \leq t+\alpha} \bigg( \frac{\partial u}{ x_{k+1}}-\frac{\partial u}{x_{k}}\bigg)(r,z)\geq \varepsilon,
\]
and
\[
\inf\limits_{t\leq r \leq t+\alpha,|z-x|\leq \alpha} \mu(r,z) \geq \varepsilon.
\]
Denote a stopping time by $\hat\tau$
\[
\hat\tau:=\min\left\{\inf\Big\{r>t: \alpha \leq |\widetilde{X}^{t,x}(r)-x|\Big\}, \inf\Big\{r>t:\widetilde{X}^{t,x}(r)\in F_{j},j\neq k\Big\}, \Big(t+\alpha\Big)\right\}.
\]
Applying the same technique in \textbf{Case 1},
\begin{equation}
\begin{aligned}
u(\hat\tau,\widetilde{X}^{t,x}(\hat\tau))-\varphi(\hat\tau,\widetilde{X}^{t,x}(\hat\tau))&=\int_{t}^{\hat\tau}[\hat M(a)-h(a)\cdot
(\hat{Z}^{t,x}(a))-\sigma\triangledown \varphi(a,\widetilde{X}^{t,x}(a))]da-(K(\hat \tau)-K(t))\\
&+\int_{t}^{\hat\tau}(\hat{Z}^{t,x}(a)-\sigma \triangledown
\varphi(a,\widetilde{X}^{t,x}(a)))\cdot d\beta(a)\\
&+\frac{1}{2}\int_{t}^{\hat\tau}\bigg(\frac{\partial \varphi}{\partial x_{k+1}}(a,\widetilde{X}^{t,x}(a))-\frac{\partial \varphi}{\partial x_{k}}(a,\widetilde{X}^{t,x}(a))\bigg)d\Lambda_{k,k+1}(a).
\end{aligned}
\end{equation}
Taking expectation on both sides and applying similar techniques as in \textbf{Case 1} to  the first three terms on the right,
\[
\varepsilon\widetilde{\mathbb E}(\hat\tau-t)+\widetilde{\mathbb E}\left\{\frac{1}{2}\int_{t}^{\hat\tau}\bigg(\frac{\partial \varphi}{\partial x_{k+1}}(a,\widetilde{X}^{t,x}(a))-\frac{\partial \varphi}{\partial x_{k}}(a,\widetilde{X}^{t,x}(a))\bigg)d\Lambda_{k,k+1}(a)\right\} \leq 0.
\]
Moreover, by
\[
\widetilde{\mathbb E}\left\{\int_{t}^{\hat\tau}\bigg(\frac{\partial \varphi}{\partial x_{k+1}}(a,\widetilde{X}^{t,x}(a))-\frac{\partial \varphi}{\partial x_{k}}(a,\widetilde{X}^{t,x}(a))\bigg)d\Lambda_{k,k+1}(a)\right\} \geq \widetilde{\mathbb E} \left\{\varepsilon \int_{t}^{\hat\tau}d\Lambda_{k,k+1}(a)\right\} \ge 0,
\]
we have $\varepsilon\widetilde{\mathbb E}(\hat\tau-t)\leq 0$, which is a contradiction.
\end{proof}

\hspace*{1em}

\subsection{The Uniqueness of Viscosity Solution}
\hspace*{0.8em}

\begin{theorem}\label{th9}
Assume that (\textbf{A1)}-(\textbf{A3)} are satisfied. If for each positive $R$, there exists a positive function $\zeta_R$ on $[0,\infty)$ with
$\lim_{r\rightarrow 0^+} \zeta_R (r) =0$
such that
\begin{equation}\label{27}
\big|G(t,x_{1},y,z)-G(t,x_{2},y,z)\big|\leq\zeta_R\left((1+|z|)|x_{1}-x_{2}|\right),
\end{equation}
for all $t\in [0,T]$,\;$x_{1},x_{2}\in \Gamma^{n}$,\;$y\in \mathbb{R}$,\;$z\in \mathbb{R}^n$ satisfying $max\big\{|x_{1}|,|x_{2}|,|y|\big\} \leq R$.
Then \eqref{4} admits at most one viscosity solution such that
\begin{equation}\label{increasing-condition}
\lim_{|x|\rightarrow\infty}|u(t,x)|e^{-A\log^2|x|}=0,\text{ uniformly in }t\in[0,T],
\end{equation}
for some $A>0$.
\end{theorem}
 To deal with the specific features of the ranked data, there are still extra spaces needed to be illustrated.
At the beginning, define a subset of $\Pi ^{n}$ for any given positive $\alpha$:
\[
\Pi ^{n}_{\alpha}:=\{x\in \Pi ^{n}:  \min\limits_{y\;\in\;\partial\Pi ^{n}}|x-y|\geq \alpha\},
\]
and we can choose a  positive constant $\alpha_{0}$ such that $\Pi ^{n}_{\alpha}\neq \emptyset$ for any $0 \leq \alpha \leq \alpha_{0}$. From the definitions of viscosity subsolution and supersolution, we only need to consider a subsolution $u$ and a supersolution $v$, which respectively satisfy
\begin{equation}\label{upper-lower-bound}
u(t,x)\geq L(t,x),\;v(t,x)\leq U(t,x),\; (t,x)\in[0,T)\times \Pi ^{n},
\end{equation}
 and prove $u(t,x)\leq v(t,x)$ on $(0,T)\times \Pi ^{n}_{\alpha}$ for every $0\leq \alpha \leq \alpha_{0}$.

Let us describe the space $\Pi ^{n}_{\alpha}$ more clearly. Define its boundary $O_{j}$:
\[
O_{j}=\mathop{\cup}\limits_{1\leq k_{1}<\dots <k_{j}\leq n-1}\mathop{\cap}\limits_{l=1}^{j}\mathcal{I}_{k_{l}},\;j=1,\cdots,n-1,
\]
with
\[
\mathcal{I}_{k}:=\left\{ x=(x_{1},x_{2},x_{3}
\cdots,x_{n})\in \Pi ^{n}_{\alpha}:x_{k+1}=x_{k}-\sqrt{2}\alpha \right\},\;k=1,\cdots,n-1.
\]
Denote the inner of $\Pi ^{n}_{\alpha}$ by $In(\Pi ^{n}_{\alpha})$, thus
\[
\Pi ^{n}_{\alpha}= \bigg(In(\Pi ^{n}_{\alpha})\bigg)\cup\bigg(\mathop{\cup}\limits_{j=1}^{n-1}O_{j}\bigg).
\]

Before the proof of  \textit{Theorem 6}, let us give the following two lemmas.
\\

\begin{lemma}\label{lemma-difference}
  Suppose that $u(\cdot,\cdot)$ and $v(\cdot,\cdot)$ are the viscosity subsolution and supersolution of ($\ref{4}$) separately, such that $u(t,x)\geq L(t,x)$ and $v(t,x) \leq U(t,x)$ for  all $(t,x)\in[0,T)\times \Pi ^{n}$.Then $w(\cdot,\cdot):=u(\cdot,\cdot)-v(\cdot,\cdot)$ is a viscosity subsolution
  of the following equation,
\begin{eqnarray}
\label{4.2.1}
\left\{\begin{array}{l}
\min\Big\{w(t,x),-\frac{\partial w}{\partial t}(t,x)-
\mathcal{L}w(t,x)-\big(c|w|+c|\sigma\triangledown w |\big)(t,x)\Big\}=0, \;  (t,x)\in [0,T)\times In(\Pi ^{n}_{\alpha}),\\
\sum\limits_{l=1}^{j} \Big(\frac{\partial w}{\partial x_{k_{l}}}(t,x)-\frac{\partial w}{\partial x_{k_{l}+1}}(t,x)\Big)=0, \quad\qquad  \qquad \qquad \qquad  \qquad  \qquad (t,x)\in[0,T)\times O_{j},\;j=1,\cdots,n-1,\\
w(T,x)=0,\qquad \qquad \qquad \qquad \qquad \qquad  \qquad \qquad  \qquad \qquad \qquad  \;\; x\in\Pi ^{n}_{\alpha}.
 \\
\end{array}\right.
\end{eqnarray}
\end{lemma}

\begin{proof}Let $\varphi \in C^{1,2}_{b}([0,T] \times \Pi ^{n}_{\alpha})$ and $w-\varphi$ achieves a local maximum point at $(\hat t,\hat x)\in [0,T)\times \Pi ^{n}_{\alpha}$ with $\varphi(\hat t,\hat x)-w(\hat t,\hat x)=0$. Without loss of generality, it is reasonable for us to assume that $(\hat t,\hat x)$ is a strict global maximum point of the function $w-\varphi$ and $\varphi$ can be modified if it is needed. The ultimate purpose is to prove that,
\begin{equation}\label{4.5-1}
\min\left\{w(\hat t,\hat x), \Big(-\frac{\partial \varphi}{\partial t}(\hat t,\hat x)-
\mathcal{L}\varphi(\hat t,\hat x)-\big(c|w|+c|\sigma\triangledown \varphi|\big)(\hat t,\hat x)\Big)\right\} \leq 0.
\end{equation}
By the definitions of $u$ and $v$, we observe that $u(t,x)\leq U(t,x)$ and $v(t,x) \geq L(t,x)$. With the assumptions in Lemma \ref{lemma-difference}, $u$ and $v$ take values in $[L(t,x), U(t,x)]$. Furthermore,  if either  $v(\hat t,\hat x)= U(\hat t,\hat x)$ or $u(\hat t,\hat x)= L(\hat t,\hat x)$, it is obviously to obtain that $w(\hat t,\hat x) \leq 0$ and \eqref{4.5-1} exists. Therefore, we just need to consider
\begin{eqnarray}
\label{4.2.2}
u(\hat t,\hat x)> L(\hat t,\hat x) \quad and \quad U(\hat t,\hat x)> v(\hat t,\hat x).
\end{eqnarray}
Now for a given positive $\varepsilon$, define
\[\mathcal{Q}_{\varepsilon}(t,x,z):=(u-\varphi)(t,x)-v(t,z)-\left(\frac{|x-z|}{\varepsilon}\right)^{2}\quad and \quad \Pi ^{n,R}_{\alpha}:= \Pi ^{n}_{\alpha}\cap S_{R},\]
where $R$ is chosen big enough such that the point $(\hat t,\hat x)$ belongs to $[0,T) \times \Pi ^{n,R}_{\alpha}$ and $S_{R}$ is an open ball in $\mathbb{R}^{n}$ with the origin as the center and radius of $R$. Let a point $(\hat t_{\varepsilon},\hat x_{\varepsilon}, \hat z_{\varepsilon})$ satisfy:
 \[
 (\hat t_{\varepsilon},\hat x_{\varepsilon}, \hat z_{\varepsilon})\ \in \operatorname{argmax} \limits_{[0,T) \times\Pi ^{n,R}_{\alpha}}\mathcal{Q}_{\varepsilon}(t,x,z).
 \]
Then applying the classical method involved in  \cite[Proposition 3.7]{Crandall}, we get

\begin{enumerate}[(i)]
\setlength{\parsep}{0ex} 
\setlength{\itemsep}{0ex}
\item $(\hat t_{\varepsilon},\hat x_{\varepsilon}, \hat z_{\varepsilon})\stackrel{\varepsilon \rightarrow 0}{\longrightarrow} (\hat t,\hat x,\hat x) $;
\item $\frac{|\hat x_\varepsilon- \hat z_\varepsilon|^{2}}{\varepsilon^{2}}$ $\stackrel{\varepsilon \rightarrow 0}{\longrightarrow} 0 $.
\end{enumerate}

On the basis of \cite[Theorem 8.3]{Crandall}, for a fixed positive $\varepsilon$ and any $\eta >0$, there exists $(X_{\eta}, Z_{\eta}) \in S^{n} \times S^{n}$ and $c_{\eta} \in \mathbb{R}$ satisfing
\begin{equation*}
\begin{aligned}
\Big(c_{\eta}+\frac{\partial \varphi }{\partial t}( \hat t_{\varepsilon},\hat x_{\varepsilon}),q_{\varepsilon}+\triangledown
\varphi( \hat t_{\varepsilon}, \hat x_{\varepsilon}),X_{\eta}\Big)\in \bar{\mathcal K}_{u}^{2,+}(\hat t_{\varepsilon},\hat  x_{\varepsilon})\quad and \quad
\Big(c_{\eta},q_{\varepsilon},Z_{\eta}\Big)\in \bar{\mathcal K}_{v}^{2,-}(\hat t_{\varepsilon},\hat z_{\varepsilon}),
\end{aligned}
\end{equation*}
and
\begin{equation*}
\begin{aligned}
\begin{pmatrix}
X_{\eta}& 0\\
0 & -Z_{\eta}
\end{pmatrix} \leq H+\eta H^{2},
\end{aligned}
\end{equation*}
where $q_{\varepsilon}:=\frac{2( \hat x_{\varepsilon}-\hat z_{\varepsilon})}{\varepsilon^{2}}$, $H:=\begin{pmatrix} D^{2}\varphi(\hat t_{\varepsilon},\hat x_{\varepsilon})+2/\varepsilon^{2}& -2/\varepsilon^{2}\\
-2/\varepsilon^{2}& 2/\varepsilon^{2}
\end{pmatrix} $.

We can get that
\begin{equation*}
\begin{aligned}
H+\eta H^{2}=2\hat I/\varepsilon^{2} +\begin{pmatrix}
D^{2}\varphi(\hat t_{\varepsilon}, \hat x_{\varepsilon})& 0\\
0 & 0
\end{pmatrix} +\eta N(\varepsilon),
\end{aligned}
\end{equation*}
where
\begin{equation*}
\hat I:=\begin{pmatrix} I & -I\\
-I & I
\end{pmatrix},
\end{equation*}
and
\begin{equation*}
\begin{aligned}
N(\varepsilon)=\frac{8}{\varepsilon^{4}}\hat I+\frac{2}{\varepsilon^{2}}\begin{pmatrix}
2D^{2}\varphi(\hat t_{\varepsilon}, \hat x_{\varepsilon})+\varepsilon^{2}(D^{2}\varphi(\hat t_{\varepsilon}, \hat x_{\varepsilon}))^{2}/2& -D^{2}\varphi(\hat t_{\varepsilon}, \hat x_{\varepsilon})\\
-D^{2}\varphi(\hat t_{\varepsilon}, \hat x_{\varepsilon}) & 0
\end{pmatrix}.
\end{aligned}
\end{equation*}
It follows from $v$ and $u$ being the viscosity subsolutuion and supersolutuion of ($\ref{4}$) that
\begin{eqnarray*}
\label{4.2.3}
\left\{\begin{array}{l}
\min\Big\{u(\hat t_{\varepsilon}, \hat x_{\varepsilon})-L(\hat t_{\varepsilon}, \hat x_{\varepsilon}), \max\Big\{ u(\hat t_{\varepsilon}, \hat x_{\varepsilon})-U(\hat t_{\varepsilon}, \hat x_{\varepsilon}), -c_{\eta} -\frac{\partial \varphi}{\partial t}(\hat t_{\varepsilon}, \hat x_{\varepsilon}) -\frac{1}{2}tr(\sigma^{2}X_{\eta})\\
\qquad \qquad \qquad \qquad \qquad \qquad \qquad -\delta \cdot (q_{\varepsilon}+\triangledown
\varphi(\hat t_{\varepsilon}, \hat x_{\varepsilon})) -G\big(t,x,u(\hat t_{\varepsilon}, \hat x_{\varepsilon}),\sigma q_{\varepsilon}+\sigma\triangledown
\varphi(\hat t_{\varepsilon}, \hat x_{\varepsilon})\big)\Big\}\Big\} \leq 0,\vspace{1ex}
\\
\max\Big\{v(\hat t_{\varepsilon},\hat  z_{\varepsilon})-U(\hat t_{\varepsilon},\hat  z_{\varepsilon}),\min \Big\{ v(\hat t_{\varepsilon},\hat  z_{\varepsilon})-L(\hat t_{\varepsilon},\hat  z_{\varepsilon}),-c_{\eta} -\frac{1}{2}tr(\sigma^{2}Z_{\eta})-\delta \cdot q_{\varepsilon}\\
\qquad \qquad \qquad \qquad\qquad \qquad \qquad-G\big(t,z,v(\hat t_{\varepsilon},\hat  z_{\varepsilon}),\sigma q_{\varepsilon}\big) \Big\} \Big\}\geq 0. \\
\end{array}\right.
\end{eqnarray*}
Note that $(\hat t_{\varepsilon},\hat x_{\varepsilon},\hat z_{\varepsilon})\rightarrow (\hat t,\hat x, \hat x) \;as \; \varepsilon \rightarrow 0$ and due to continuity of the $L,U ,u,v$, there exists sufficiently small $\varepsilon>0 $ such that  $(\hat t_{\varepsilon}, \hat x_{\varepsilon}) $ and $(\hat t_{\varepsilon},\hat  z_{\varepsilon}) $ satisfy ($\ref{4.2.2}$), i.e., $$L(\hat t_{\varepsilon}, \hat x_{\varepsilon})<u(\hat t_{\varepsilon}, \hat x_{\varepsilon})\leq U(\hat t_{\varepsilon}, \hat x_{\varepsilon}),\qquad L(\hat t_{\varepsilon}, \hat z_{\varepsilon})\leq v(\hat t_{\varepsilon}, \hat z_{\varepsilon})<U(\hat t_{\varepsilon}, \hat z_{\varepsilon}).$$ Therefore,
\begin{eqnarray*}
\label{4.2.4}
\left\{\begin{array}{l}
-c_{\eta} -\frac{\partial \varphi}{\partial t}(\hat t_{\varepsilon}, \hat x_{\varepsilon}) -\frac{1}{2}tr(\sigma^{2}X_{\eta})-\delta \cdot (q_{\varepsilon}+\triangledown
\varphi(\hat t_{\varepsilon}, \hat x_{\varepsilon}) ) -G\big(\hat t_{\varepsilon}, \hat x_{\varepsilon},u(\hat t_{\varepsilon}, \hat x_{\varepsilon}),\sigma q_{\varepsilon}+\sigma\triangledown
\varphi(\hat t_{\varepsilon}, \hat x_{\varepsilon})\big)\leq 0,
\\
-c_{\eta} -\frac{1}{2}tr(\sigma^{2}Z_{\eta})-\delta \cdot q_{\varepsilon}
-G\big(\hat t_{\varepsilon}, \hat z_{\varepsilon},v(\hat t_{\varepsilon}, \hat z_{\varepsilon}),\sigma q_{\varepsilon}\big)\geq 0.
\end{array}\right.
\end{eqnarray*}
By above two inequalities,
\begin{equation*}
\begin{aligned}
&\frac{\partial \varphi}{\partial t} (\hat t_{\varepsilon}, \hat x_{\varepsilon})+\frac{1}{2}(tr(\sigma^{2}X_{\eta})-tr(\sigma^{2}Z_{\eta}))+\delta \cdot \triangledown
\varphi (\hat t_{\varepsilon}, \hat x_{\varepsilon} )\\&+G\big(\hat t_{\varepsilon},\hat x_{\varepsilon},u(\hat t_{\varepsilon}, \hat x_{\varepsilon}),\sigma(q_{\varepsilon}+\triangledown
\varphi(\hat t_{\varepsilon}, \hat x_{\varepsilon}))\big)-G\big(\hat t_{\varepsilon},\hat z_{\varepsilon},v(\hat t_{\varepsilon}, \hat z_{\varepsilon}),\sigma q_{\varepsilon}\big) \geq 0.
\end{aligned}
\end{equation*}
Considering the first part of the above equation, it is easy to find
\begin{equation*}
\frac{1}{2}(tr(\sigma^{2}X_{\eta})-tr(\sigma^{2}Z_{\eta})) \leq \frac{1}{2}tr(\sigma^{2}D^{2}\varphi(\hat t_{\varepsilon}, \hat x_{\varepsilon}))+\frac{\eta}{2}P_{\varepsilon},
\end{equation*}
in which
\[P_{\varepsilon} :=  \Bigg \langle
\begin{pmatrix}
\sigma\\
\sigma
\end{pmatrix} , N(\varepsilon)\begin{pmatrix}
\sigma\\
\sigma
\end{pmatrix}
\Bigg \rangle . \]
For the second part, $\zeta$ is used as the modulus $\zeta_{R}$ in ($\ref{27}$) when $R$ tends to infinity,
\begin{equation*}
\begin{aligned}
&G\big(\hat t_{\varepsilon},\hat x_{\varepsilon},u(\hat t_{\varepsilon}, \hat x_{\varepsilon}),\sigma q_{\varepsilon}+\sigma \triangledown
\varphi(\hat t_{\varepsilon}, \hat x_{\varepsilon}))\big)-G\big(\hat t_{\varepsilon},\hat z_{\varepsilon},v(\hat t_{\varepsilon}, \hat z_{\varepsilon}),\sigma q_{\varepsilon}\big)
\\&=G\big(\hat t_{\varepsilon},\hat x_{\varepsilon},u(\hat t_{\varepsilon}, \hat x_{\varepsilon}),\sigma q_{\varepsilon}+\sigma\triangledown
\varphi(\hat t_{\varepsilon}, \hat x_{\varepsilon})\big)-G\big(\hat t_{\varepsilon},\hat x_{\varepsilon},v(\hat t_{\varepsilon}, \hat z_{\varepsilon}),\sigma q_{\varepsilon}\big)
\\&+G\big(\hat t_{\varepsilon},\hat x_{\varepsilon},v(\hat t_{\varepsilon}, \hat z_{\varepsilon}),\sigma q_{\varepsilon}\big) -G\big(\hat t_{\varepsilon},\hat z_{\varepsilon},v(\hat t_{\varepsilon}, \hat z_{\varepsilon}),\sigma q_{\varepsilon}\big)
\\
&\leq c|u(\hat t_{\varepsilon}, \hat x_{\varepsilon})-v(\hat t_{\varepsilon}, \hat z_{\varepsilon})|+c|\sigma \triangledown
\varphi(\hat t_{\varepsilon}, \hat x_{\varepsilon})|+\zeta(|\hat x_{\varepsilon}-\hat z_{\varepsilon}|(|\sigma q_{\varepsilon}|+1)).
\end{aligned}
\end{equation*}
Letting $\eta \rightarrow 0$ and $\varepsilon \rightarrow 0$,
\[
-\frac{\partial \varphi}{\partial t}(\hat t,\hat x)-\delta \cdot \triangledown
\varphi(\hat t,\hat x) -c|w(\hat t,\hat x)|-c|\sigma \triangledown
\varphi(\hat t,\hat x)|-\frac{1}{2}tr(\sigma^{2}D^{2}\varphi(\hat t,\hat x)) \leq 0 .
\]
Therefore,
\[
\min\left\{w(\hat t,\hat x), \left(-\frac{\partial \varphi}{\partial t}(\hat t,\hat x)-
\mathcal{L}\varphi(\hat t,\hat x)-c|w(\hat t,\hat x)|-c|\sigma\triangledown \varphi (\hat t,\hat x)|\right) \right\}\leq 0.
\]
\end{proof}

The next lemma has been proved in \cite{feng1}.

\begin{lemma}\label{exp}
 For any fixed positive constant $A$, there is a positive $C$ and a function $\phi (x):=(\frac{1}{2}\log(1+|x|^{2})+1)^{2}$ such that
 \begin{equation}
 \label{psi}
 \psi(t,x):=e^{(A+C(T-t))\phi (x)},
 \end{equation}
satisfies
\begin{equation*}
-\frac{\partial \psi}{\partial t}(t,x)-
\mathcal{L}\psi(t,x)-c\big(\psi(t,x)+|\sigma \triangledown \psi(t,x) |\big)>0, \quad (t,x)\in [t_{1},T]\times \Pi^{n},
\end{equation*}
in which $c$ is the Lipschitz constant of $G$ and $t_{1}:=(T-\frac{A}{C})^{+}$.
\end{lemma}
\hspace*{1em}

\textbf{Proof of Theorem 6:} The purpose is to prove that $w\leq 0 $ on $[0,T]\times \Pi^{n}_{\alpha}$ with $\alpha$ varying from $0$ to $\alpha_{0}$.
According to \eqref{increasing-condition}, we have
\[
\lim\limits_{|x|\rightarrow \infty}e^{-A(\frac{1}{2}\log(1+|x|^{2})+1)^{2}}|w(t,x)|=0.
\]
Now define the following function:
\begin{equation*}
N(\rho, T):=\max\limits_{[t_{1},T]\times \Pi^{n}_{\alpha}}\big(w(t,x)-\rho \psi(t,x)\big)e^{c(t-T)},\ for \ any \ \rho >0 .
\end{equation*}
Noticing  the fact that $|w(t,x)|<\rho e^{A\big(\frac{1}{2}\log(1+|x|^{2})+1\big)^{2}} \leq \rho\psi(t,x)$, $N(\rho, T)$ is achieved at point $(t^{*},x^{*})$.

We now need to show that $N(\rho, T) \leq 0$ for all positive $\rho$. Suppose that if $t^{*}=T$, $N(\rho, T) $ is no greater than $0$ owing to $w(T,x)\leq 0$.

As for the other situation  $t^{*}<T$, suppose $N(\rho, T) > 0$. We can get
\[
w(t^{*},x^{*})=\rho \psi(t^{*},x^{*})+e^{-c(t^{*}-T)}N(\rho, T)>0.
\]
Furthermore, let us set
\[
\mathcal{H}(t,x):=N(\rho,T)e^{c(T-t)}+\rho \psi(t,x).
\]
Obviously, $\mathcal{H}(t^{*},x^{*})=w(t^{*},x^{*})>0$ and
\begin{equation*}
w(t,x)-\mathcal{H}(t,x) \leq 0, \;(t,x) \in [t_{1},T]\times \Pi^{n}_{\alpha}.
\end{equation*}
Since $w$ is a viscosity subsolution of ($\ref{4}$), for the case $x^{*} \in In(\Pi^{n}_{\alpha})$,
\[
\min\left\{w(t^{*},x^{*}), \left(-\frac{\partial\mathcal{H}}{\partial t}(t^{*},x^{*})-
\mathcal{L}\mathcal{H}(t^{*},x^{*})-cw(t^{*},x^{*})-c|\sigma\triangledown \mathcal{H}(t^{*},x^{*})|\right)\right\}\leq 0 ;
\]
for the case when $x^{*} \in O_{j}$ for some $j=1,...,n-1$,
\begin{equation}\label{H}
\begin{aligned}
\min\Bigg\{\Bigg(\sum_{l=1}^{j}&\left( \frac{\partial \mathcal{H}}{\partial x_{k_{l}} }(t^{*},x^{*})  -\frac{\partial \mathcal{H}}{\partial x_{k_{l}+1} }(t^{*},x^{*}) \right )\Bigg),  w(t^{*},x^{*})
, \\
&\left(-\frac{\partial \mathcal{H}}{\partial t}(t^{*},x^{*})-
\mathcal{L}\mathcal{H}(t^{*},x^{*})-cw(t^{*},x^{*})-c|\sigma\triangledown \mathcal{H}(t^{*},x^{*})|\right)\Bigg\} \leq 0.
\end{aligned}
\end{equation}
For $x^{*} \in O_{j}$, by some routine calculations we can see that the first partition in ($\ref{H}$) is greater than $0$.
Consequently, it can be
\[
-\frac{\partial \mathcal{H}}{\partial t}(t^{*},x^{*})-
\mathcal{L}\psi(t^{*},x^{*})-cw(t^{*},x^{*})-c|\sigma\triangledown \mathcal{H}(t^{*},x^{*})| \leq 0.
\]
It is equivalent with
\[
\rho \left(-\frac{\partial \psi}{\partial t}(t^{*},x^{*})-
\mathcal{L}\psi(t^{*},x^{*})-cw(t^{*},x^{*})-c|\sigma\triangledown \psi(t^{*},x^{*})| \right)\leq 0,
\]
which is  contrary  to the Lemma \ref{exp}. Hence $N(\rho, T)\leq 0$ with $t^{*}<T$. By the definition of $N(\rho, T)$, we obtain $w\leq 0$. Repeating the same discussion on $[t_{i},t_{i-1}]$ for $i=1,2,\cdots$, where $t_{i+1}:=(t_{i}-\frac{A}{C})^{+}$. Finally, we obtain $w\leq 0$ for every point $(t,x)$ in $[0,T] \times \Pi^{n}_{\alpha}$ after going through the above processes for at most $\frac{CT}{A}$ steps. \qed

\section{Dynkin Game and American Game Option}

In this part, we will study the application in American game option with  ranked data based on  Dynkin game. In this practical problem, the following spaces are given according to the aforementioned definitions:
\begin{itemize}
\setlength{\itemsep}{0pt}
\item $\Pi^{n,+}:\{x\in\mathbb{R}^{n}:0<x_{n}<x_{n-1}<\cdots<x_{1}\}$.
\item $F_{k}^{+}:=\{x\in\partial\Pi^{n,+}: 0<x_{n}<x_{n-1}<\cdots<x_{k+1}=x_{k}<\cdots<x_{1}<0 \},\;k=1,\cdots,n-1$.
\item $\Gamma^{n,+}:=F_{1}^{+}\cup F_{2}^{+} \cup \cdots \cup F_{n-1}^{+}\cup \Pi^{n,+}$.
\end{itemize}

We consider a financial market containing $n$ stocks and one bond. More precisely, the values $\{P_{0}^{t, \hat p},P_{k}^{t, \hat p},k=1,\cdots,n\}$ of the assets satisfy:

\begin{equation}
\label{5}
\left\{\begin{aligned}
P_{0}^{t, \hat p}(a)=& p_{0}+\int_{t}^{a} r(s) P_{0}^{t, \hat p}(s) \mathrm{d}s, \\
P_{k}^{t, \hat p}(a)=& p_{k}+\int_{t}^{a} P_{k}^{t, \hat p}(s)\left(\sum_{i=1}^{n} \delta_{i}\textbf{1}_{\left\{P_{k}^{t, \hat p }(s)=P_{(i)}^{t, \hat p}(s)\right\}} \mathrm{~d} s\right.\left.+\sum_{i=1}^{n}\sigma_{i} \textbf{1}_{\left\{P_{k}^{t,\hat p}(s)=P_{(i)}^{t,\hat p}(s)\right\}} d W_{k}(s)\right) ,
\end{aligned}\right.
\end{equation}
where $a\geq t$, $\hat p:=(p_{0},p_{1},\cdots,p_{n})\in \mathbb{R^{+}}\times\Gamma^{n,+}$ and $\{\sigma^{2}_{1},\cdots,\sigma^{2}_{n}\}$  satisfies $\sigma_{i+1}^{2}\geq \frac{1}{2}(\sigma_{i}^{2}+\sigma_{i+2}^{2})$ for $i=1,\ldots,n-2$.
Moreover, based on the notations in \textit{Section 2}, the ranked price evolves following:
\[
d P_{(i)}^{t, \hat p}(a)=P_{(i)}^{t, \hat p}(a)\left(\delta_{i} d r+\sigma_{i} d \beta_{i}(a)+\frac{1}{2}d \Lambda_{i, i+1}(a)-\frac{1}{2}d \Lambda_{i-1, i}(a)\right), i=1, 2,3,\ldots, n.
\]
More detailed explanations can be found in \cite{feng2}. For simplicity, let us assume that the interest rate $r(s)$ is a constant in the subsequence context and denote it as $r_{0}$. The relevant results can be trivially extended to the situation that $r(s)$ is a deterministic function.

\hspace*{1em}
\subsection{Dynkin Game}
\hspace*{1em}

Before looking at the American game option, firstly we explore the Dynkin game. Suppose there are two players in the market and both of them have the right to choose a stopping time to terminate the game, i.e., $\lambda$, $\tau$ respectively, which are $\{\mathcal F_t^W\}$-stopping times. At the stopping time, player $A$ needs to pay for player $B$ with the following amounts:
\begin{equation}
\label{5.1}
R_{t}(\lambda,\tau):=\xi\textbf{1}_{\{\lambda \wedge \tau = T\}}+U(\lambda)\textbf{1}_{\{\lambda < \tau \}}+L(\tau)\textbf{1}_{\{\tau<T,\tau \leq \lambda\}}.
\end{equation}
That is, suppose that player $A$ stops the game first, i.e., $\lambda < \tau $, $A$ should pay $U(\lambda)$ to $B$. If $B$ stops first, $A$ pays $L(\tau)$ to $B$. If no player stops before the time $T$, $A$ will pay $\xi$ to $B$. Thus, player $A$ wants to minimize his/her cost, nevertheless, player $B$ wants to maximize his/her reward. Define the upper and lower values,
\begin{equation*}
\begin{aligned}
&\underline{V}(t):=\mathop{esssup}\limits_{\tau \in \mathcal{M}_{t,T}} \mathop{essinf}\limits_{\lambda \in \mathcal{M}_{t,T}}\mathbb E\left[R_{t}(\lambda,\tau) \mid \mathcal{F}_{t}\right],\\
&\overline{V}(t):= \mathop{essinf}\limits_{\lambda \in \mathcal{M}_{t,T}} \mathop{esssup}\limits_{\tau \in \mathcal{M}_{t,T}}\mathbb E\left[R_{t}(\lambda,\tau) \mid \mathcal{F}_{t}\right],
\end{aligned}
\end{equation*}
 in which $ \mathcal{M}_{t,T}$ represents the class of $\{\mathcal F_t^W\}$-stopping times valuing in $[t,T]$. If $V:=\underline{V}=\overline{V}$, $V(t)$ is referred to be the value of the above Dynkin game.

Considering this following simplified equations:
\begin{eqnarray}
\label{5.1.1}
\left\{\begin{array}{l}
Y(t)=g(\widetilde{X}(T))+\displaystyle\int_{t}^{T}G(s,\widetilde{X}(s),Y(s),\hat{Z}(s))ds -\int_{t}^{T}\hat{Z}(s)
\cdot d\beta(s)+K(T)-K(t),\\
\displaystyle L(t,\widetilde{X}(t))\leq Y(t)\leq U(t,\widetilde{X}(t)),\quad t\in[0,T],\\
\displaystyle\int_{0}^{T}\big[Y(s)-L(s,\widetilde{X}(s))\big]dK^{+}(s)=\displaystyle\int_{0}^{T}\big[ U(s,\widetilde{X}(s))-Y(s)\big]dK^{-}(s)=0.
\end{array}\right.
\end{eqnarray}

Under assumptions $\textbf{(A1)}-\textbf{(A3)}$, we know that \eqref{5.1.1} admits a unique solution.
By \cite[Theorem 4.1]{Karatzas}, we have the following result.
\begin{theorem}
Under assumptions $\textbf{(A1)}-\textbf{(A3)}$, the game (\ref{5.1}) with $\xi := g(\widetilde{X}(T))$, $L(t):=L(t,\widetilde{X}(t))$ and $U(t):=U(t,\widetilde{X}(t))$ has the value $V(t)$, given by backward component $Y$ of the solution to \eqref{5.1.1} for all $0\leq t\leq T$, i.e.,
\[
V(t):=\underline{V}(t)=\overline{V}(t)=Y(t),\; a.s..
\]
Futhermore,  there exists a saddle-point $(\hat\lambda_{t},\hat\tau_{t})\in \mathcal{M}_{t,T} \times \mathcal{M}_{t,T}$,
\[
\hat\lambda_{t}:=\min\Big\{T,\inf\Big\{a\in[t,T)|Y(a)=U(a)\Big\}\Big\}, \]
\[
\hat\tau_{t}:=\min\Big\{T,\inf\Big\{a\in[t,T)|Y(a)=L(a)\Big\}\Big\},
\]
that is,
\[
\mathbb E\left[R_{t}(\hat\lambda_{t},\tau ) \mid \mathcal{F}_{t}\right]\leq\mathbb E\left[R_{t}(\hat\lambda_{t},\hat\tau_{t} ) \mid \mathcal{F}_{t}\right]=Y(t)\leq\mathbb E\left[ R_{t}(\lambda,\hat\tau_{t})\mid \mathcal{F}_{t}\right],\; a.s.,
\]
for point $(\lambda,\tau)$ in $\mathcal{M}_{t,T} \times \mathcal{M}_{t,T}$.
\end{theorem}

\hspace*{1em}
\subsection{American Game Option}
\hspace*{1em}

 Assume that there are two participants, a broker and a trader. The American game option means that the trader needs to pay a premium and is entitled to ask for a contingent claim when he (or she) decides to exercise while the broker can cancel the contract at any time at the expense of a higher amount. When the trader first exercises at time $\sigma$, he (or she) gets $L(\sigma,\widetilde{X}(\sigma))$. If the broker choose to first cancel the contract at $\lambda$, he (or she) needs to pay $U(\lambda,\widetilde{X}(\lambda))$, which is a larger amount than $L(\lambda,\widetilde{X}(\lambda))$. The difference between the amount of $U(\lambda,\widetilde{X}(\lambda))$ and $L(\lambda,\widetilde{X}(\lambda))$ is a penalty. If neither exercises the option by the maturity date $T$, then the broker pays $g(\widetilde{X}(T))$. In conclusion, if the broker decides to exercise at $\lambda\leq T$ and the trader exercises at $\sigma\leq T$, then the broker pays to the trader $R(\lambda,\sigma)$:
 \[
R(\lambda,\sigma):=\textbf{1}_{\{\sigma \wedge \lambda = T\}}g(\widetilde{X}(T))+\textbf{1}_{\{\sigma \leq \lambda\}}L(\sigma,\widetilde{X}(\sigma))+\textbf{1}_{\{\lambda < \sigma\}}U(\lambda,\widetilde{X}(\lambda)).
\]
 It can be seen that this is a typical applicant of Dynkin games and the relevant definitions are introduced as follows.

\begin{definition}
 A self-financing portfolio on $[t,T]$ combined with initial endowment $H$ (at initial time) is a process $\pi=(\pi_{1},\cdots,\pi_{n})$  so as to make the wealth process $Y^{\pi,H}$ meet
\begin{equation*}
Y^{\pi,H}(s)=H+\int_{t}^{s}\sum_{k=1}^{n}\pi_{k}(u)\frac{dP_{k}(u)}{P_{k}(u)}+\int_{t}^{s}\left(Y^{\pi,H}(u)-\sum_{k=1}^{n}\pi_{k}(u)\right)\frac{dP_{0}(u)}{P_{0}(u)},\;\forall s\in[t,T].
\end{equation*}
\end{definition}

\begin{definition}
A hedge strategy $(\pi,\lambda)$ against the option on $[t,T]$ with return
\[
R(\lambda,s):=\textbf{1}_{\{s \wedge \lambda = T\}}g(\widetilde{X}(T))+\textbf{1}_{\{s\leq \lambda\}}L(s,\widetilde{X}(s))+\textbf{1}_{\{\lambda < s \}}U(\lambda,\widetilde{X}(\lambda)),\quad s\in[t,T],
\]
is a pair of a self-financing portfolio $\pi$ with initial endowment $H$ at $t$ and an $\{\mathcal F_t^W\}$-stopping time $\lambda$
 such that
\[
Y^{\pi,H}(s\wedge \lambda)\geq R(\lambda,s).
\]

\end{definition}

\begin{definition}
 The fair value at time $t$ of the option is the infimum of initial endowment $H$ for which a hedge exists, i.e.,
\[
V(t):= \inf\Big\{H\geq0:there\;is\;a\;hedge\;strategy\;(\pi,\lambda)\Big\}.
\]
\end{definition}

In this section, owing to $(\ref{5})$, the wealth process evolves:
\begin{equation}
\left\{\begin{aligned}
dY^{\pi,H}(s)&=\left[Y^{\pi,H}(s)r_{0}+\sum_{k=1}^{n} \pi_{k}(s)\left(\sum_{i=1}^{n} \left(\delta_{i}-r_{0}\right)\mathbf{1}_{\left\{P_{k}^{t, \hat p}(s)=P_{(i)}^{t, \hat p}(s)\right\}}\right)\right] ds \\
&+\sum_{k=1}^{n} \pi_{k}(s)\left(\sum_{i=1}^{n}\sigma_{i}\mathbf{1}_{\left\{P_{k}^{t, \hat p}(s)=P_{(i)}^{t, \hat p}(s)\right\}} \right) d W_{k}(s), \\
Y^{\pi, H}(t)&=H.
\end{aligned}\right.
\end{equation}
Denote
\begin{equation}
\label{5.2.2}
\bar\pi_{i}(s):=\sum_{k=1}^{n}\pi_{k}(s)\mathbf1_{\left\{P_{k}^{t, \hat p}(s)=P_{(i)}^{t, \hat p}(s)\right\}}\sigma_{i},\;i=1,2,3,\cdots,n.
\end{equation}
Thus,
\begin{equation}
dY^{\pi,H}(s)=\left(Y^{\pi,H}(s)r_{0}+\sum_{k=1}^n\frac{\delta_{k}-r_{0}}{\sigma_{k}}\bar\pi_{k}(s)\right)ds+\sum_{k=1}^{n}\bar\pi_{k}(s)d\beta_{k}(s).
\end{equation}
It follows that the existence of self-financing portfolio $\pi$ is
equivalent to the existence of $\bar \pi$.

Before presenting the main result, we give a new two-barriers-reflected-BSDE. Inducing the following new measure
\[
d\mathbb P^{*}=\exp\left\{\int_{0}^{t}-\frac{1}{2}\sum_{j=1}^{n}\left(\frac{\delta_{j}-r_{0}}{\sigma_{j}}\right)^{2}du-\sum_{j=1}^{n}\int_{0}^{t}\left(\frac{\delta_{j}-r_{0}}{\sigma_{j}}\right)d\beta_{j}(u)\right\}d\mathbb P,
\]
then by the Girsanov Theorem,
\[\widetilde{\beta}_{j}(t)=\beta_{j}(t)+\int_{0}^{t}\left(\frac{\delta_{j}-r_{0}}{\sigma_{j}}\right)du,\;j=1,\cdots,n,\]
is an  $(\mathcal{F}_{t},\mathbb P^{*})$-Brownian motion. 
Introduce this following BSDE on space $(\Omega,\mathcal{F},\mathbb P^{*})$ for $t$ in $[0,T]$:
\begin{eqnarray}
\label{5.2}
\left\{\begin{array}{l}
\widetilde{Y}(t)=g(\widetilde{X}(T))e^{-r_{0}T}-\displaystyle \int_{t}^{T}\hat{\widetilde{Z}}^{t,x}(u)\cdot d\widetilde{\beta}(u) +\widetilde{K}(T)-\widetilde{K}(t),\vspace{1ex}\\
L(t)e^{-r_{0}t}\leq \widetilde{Y}(t)\leq U(t)e^{-r_{0}t},\vspace{1ex}\\
\displaystyle \int_{0}^{T}\left(U(u)e^{-r_{0}u}-\widetilde{Y}(u)\right)d\widetilde{K}^{-}(u)=\int_{0}^{T}\left(\widetilde{Y}(u)-L(u)e^{-r_{0}u}\right)d\widetilde{K}^{+}(u)=0,
\end{array}\right.
\end{eqnarray}
where $\widetilde{\beta}$ is the vector of $\{\widetilde{\beta}_{j}\}$. It then follows from  \textit{Theorem 1} that we can suppose $(\widetilde{Y},\widetilde{Z},\widetilde{K}^{+},\widetilde{K}^{-})$ is a unique solution of ($\ref{5.2}$). Taking expectation $\mathbb{E}^{*}$ and based on  \textit{Theorem 9}, we also have
\[
\mathop{esssup}\limits_{\tau \in \mathcal{M}_{t,T}} \mathop{essinf}\limits_{\lambda \in \mathcal{M}_{t,T}}\mathbb E^{*}\left[R_{t}^1(\lambda,\tau) \mid \mathcal{F}_{t}\right]=\widetilde{Y}(t)= \mathop{essinf}\limits_{\lambda \in \mathcal{M}_{t,T}}\mathop{esssup}\limits_{\tau \in \mathcal{M}_{t,T}}\mathbb  E^{*}\left[R_{t}^1(\lambda,\tau) \mid \mathcal{F}_{t}\right],
\]
where $R^{1}_{t}(\lambda,\tau)=g(\widetilde{X}(T))\textbf{1}_{\{\tau \wedge \lambda = T\}}e^{-r_{0}T}+\textbf{1}_{\{\tau < T, \tau \leq \lambda \}}L(\tau)e^{-r_{0}\tau}+\textbf{1}_{\{\lambda < \tau\}}U(\lambda)e^{-r_{0}\lambda}$.

Define the stopping time again
\begin{equation*}
\begin{aligned}
&\hat\tau_{t}^{*}:=\min\Big\{T,\inf\Big\{a\in[t,T)\mid \widetilde{Y}(a)=e^{-r_{0}t}L(a)\Big\}\Big\}.\\
&\hat\lambda_{t}^{*}:=\min\Big\{T,\inf\Big\{a\in[t,T)\mid \widetilde{Y}(a)=e^{-r_{0}t}U(a)\Big\}\Big\}.
\end{aligned}
\end{equation*}

It is ready for us to give this following Theorem.

\begin{theorem} The fair value $V(t)$ of American game option with $R(\lambda,\sigma)$ is $e^{r_{0}t}\widetilde{Y}(t)$ for $t \leq T$.
\end{theorem}

\begin{proof} Fix time $t\leq T$ and suppose $(\pi,\lambda)$ is a hedge strategy. Hence, $\pi=(\pi_{1},\cdots,\pi_{n})$ is a self-financing portfolio and  $Y^{\pi,H}(s\wedge \lambda)$ is no less than $R(\lambda,s)$ for $t\leq s\leq T$. Thus using It\^{o}'s formula for $e^{-r_{0}t}Y^{\pi,H}$ from $t$ to $s\wedge \lambda$,
\begin{equation}
\begin{aligned}
e^{-r_{0}(s\wedge\lambda)}Y^{\pi,H}(s\wedge \lambda)&=e^{-r_{0}t}H+\int_{t}^{s\wedge \lambda}e^{-r_{0}u}\sum_{j=1}^{n}\bar\pi_{j}(u)d\widetilde{\beta}_{j}(u)\geq e^{-r_{0}(s\wedge\lambda)}R(\lambda,s) ,\;s\in[t,T].
\end{aligned}
\end{equation}
For any stopping time $\tau\geq t$,
\begin{equation*}
e^{-r_{0}t}H \geq \mathbb{E}^{*}\left[R(\lambda,\tau)e^{-r_{0}(\tau \wedge\lambda)}\mid \mathcal{F}_{t} \right]
\end{equation*}
Furthermore,
\begin{equation}\label{inequality-1}
\begin{aligned}
e^{-r_{0}t}H &\geq \mathop{esssup}\limits_{\tau \in \mathcal{M}_{t,T}}\mathbb{E}^{*}\left[R(\lambda,\tau)e^{-r_{0}(\tau\wedge\lambda)}\mid \mathcal{F}_{t}\right]\\
&\geq \mathop{essinf}\limits_{\lambda \in \mathcal{M}_{t,T}} \mathop{esssup}\limits_{\tau \in \mathcal{M}_{t,T}}\mathbb{E}^{*}\left[R(\lambda,\tau)e^{-r_{0}(\tau\wedge\lambda)}\mid \mathcal{F}_{t}\right]\\
&=\mathop{essinf}\limits_{\lambda \in \mathcal{M}_{t,T}} \mathop{esssup}\limits_{\tau \in \mathcal{M}_{t,T}}\mathbb{E}^{*}\left[R_{t}^{1}(\lambda,\tau)\mid \mathcal{F}_{t}\right]\\
&=\widetilde{Y}.
\end{aligned}
\end{equation}
This means $V(t)\geq e^{r_{0}t}\widetilde{Y}(t)$. We turn to the opposite inequality. For $t \leq s \leq T$, it is easily to see that
\begin{equation*}
\begin{aligned}
\widetilde{Y}(t)+\int_{t}^{\hat{\lambda}_{t}^{*}\wedge s}\widetilde{Z}(u)\cdot d \widetilde{\beta}(u)
&\geq \textbf{1}_{\{s\leq \hat{\lambda}_{t}^{*}\}}\widetilde{Y}(s)+\textbf{1}_{\{\hat{\lambda}_{t}^{*}<s\}}\widetilde{Y}(\hat{\lambda}_{t}^{*})
+\textbf{1}_{\{s=\hat{\lambda}_{t}^{*}=T\}}g(\widetilde{X}(T))e^{-r_{0}T}\\
&\geq \textbf{1}_{\{s\leq \hat{\lambda}_{t}^{*}\}}L(s)e^{-r_{0}s}+\textbf{1}_{\{\hat{\lambda}_{t}^{*}<s\}}U(\hat{\lambda}_{t}^{*})e^{-r_{0}\hat{\lambda}_{t}^{*}}+\textbf{1}_{\{\hat{\lambda}_{t}^{*}=s=T\}}g(\widetilde{X}(T))e^{-r_{0}T}\\
&=R(\hat{\lambda}_{t}^{*},s)e^{-r_{0}(\hat{\lambda}_{t}^{*}\wedge s)}.
\end{aligned}
\end{equation*}
That is, for s varying from $t$ to $T$,
\[
\widetilde{Y}(t)+\int_{t}^{\hat{\lambda}_{t}^{*}\wedge s}\widetilde{Z}(u)\cdot d \widetilde{\beta}(u) \geq R(\hat{\lambda}_{t}^{*},s)e^{-r_{0}(\hat{\lambda}_{t}^{*}\wedge s)}.
\]
Define $\mathcal{M}(s):=\widetilde{Y}(t)e^{r_{0}s}+e^{r_{0}s}\int_{t}^{s}\textbf{1}_{\{u\leq \lambda_{t}^{*}\}}\widetilde{Z}(u)\cdot d\widetilde{\beta}(u)$ and by It\^{o}'s formula
\[
\mathcal{M}(s)=e^{r_{0}t}\widetilde{Y}(t)+\displaystyle\int_{t}^{s}e^{r_{0}u}\textbf{1}_{\{u\leq \lambda_{t}^{*}\}}\widetilde{Z}(u) \cdot d\widetilde{\beta}(u)+\int_{t}^{s} r_{0}\mathcal{M}(u)du.
\]
Setting $\bar\pi_{j}(u):=e^{r_{0}u}\widetilde{Z}_{j}(u)\textbf{1}_{\{u\leq \lambda_{t}^{*}\}}$
 and owing to the uniqueness relation $(\ref{5.2.2})$ between $\pi$ and $\bar\pi$, we can find a self-financing portfolio $\pi=({\pi_{1},\cdots,\pi_{n}})$  which has initial value $e^{r_{0}t}\widetilde{Y}(t)$. Also, $(\pi,\lambda_{t}^{*})$ is a hedge against the American game option since $\mathcal{M}(s\wedge \lambda_{t}^{*})\geq R(\hat{\lambda}_{t}^{*},s)$ for $s$ varying from $t$ to $T$. As a consequence,
\begin{equation}\label{inequality-2}
V(t) \leq e^{r_{0}t}\widetilde{Y}(t).
\end{equation}
Combining \eqref{inequality-1} and \eqref{inequality-2}, the fair value of the American game option is $e^{r_{0}t}\widetilde{Y}(t)$.
\end{proof}

\section*{Acknowledgement}
Xinwei Feng's work is supported by National Natural Science
Foundation of China (No. 12001317), Shandong Provincial
Natural Science Foundation (No. ZR2020QA019) and QILU Young Scholars Program of Shandong University.





\end{document}